\documentclass[11pt]{article}
\usepackage[margin=1truein]{geometry}

\usepackage{authblk}
\usepackage[utf8]{inputenc} 
\usepackage{hyperref}       
\usepackage{url}            
\usepackage{booktabs}       
\usepackage{amsmath,amssymb,amsfonts}   
\usepackage{amsthm}%
\usepackage{mathrsfs}
\usepackage{nicefrac}       
\usepackage{microtype}      
\usepackage{xcolor}         
\usepackage{pgfplots}
\pgfplotsset{compat=newest}
\usepackage{pxpgfmark} %
\usepackage{bm} 
\usepackage{subcaption}
\usepackage{algorithm}%
\usepackage{algorithmicx}%
\usepackage{algpseudocode}%
\usepackage{listings}%
\usepackage{makecell}

\usepackage{hyperref}
\usepackage{cleveref}
 
\usepackage{autonum} 

\newtheorem{theorem}{Theorem}
\newtheorem{proposition}[theorem]{Proposition}%

\newtheorem{lem}[theorem]{Lemma}
\newtheorem{asp}{Assumption}

\usepackage{wasysym}
\newcommand\doublecheck{\textcolor{blue}{\checked\kern-0.6em\checked}}

\DeclareMathOperator*{\argmin}{arg\,min}

\title{Subspace Quasi-Newton Method with \\Gradient Approximation}

\author[1]{Taisei Miyaishi}
\author[1]{Ryota Nozawa}
\author[2]{Pierre-Louis Poirion}
\author[1,2]{Akiko Takeda}
\affil[1]{Graduate School of Information Science and Technology, University of Tokyo}
\affil[2]{Center for Advanced Intelligence Project, RIKEN}

\begin{document}

\maketitle

\begin{abstract}
In recent years, various subspace algorithms have been developed to handle large-scale optimization problems. Although existing subspace Newton methods require fewer iterations to converge in practice, the matrix operations and full gradient computation are bottlenecks when dealing with large-scale problems.
We propose a subspace quasi-Newton method that is restricted to a deterministic-subspace together with a gradient approximation based on random matrix theory. Our method does not require full gradients, let alone Hessian matrices. Yet, it achieves the same order of the worst-case iteration complexities in average for convex and nonconvex cases, compared to existing subspace methods. In numerical experiments, we confirm the superiority of our algorithm in terms of computation time.
\end{abstract}

\section{Introduction}

In recent years, driven by advancements in the field of machine learning, considerable attention has been devoted to the study of unconstrained large-scale optimization problems. The problem of interest is formulated as follows:
\begin{equation}\label{problem}
	\min_{\bm{x}\in\mathbb{R}^n} f(\bm{x}) \quad \text{(assuming $f\in C^2$)}.
\end{equation}
Here, $f$ represents a non-convex function, and $n$ is assumed to be sufficiently large such that even computing the gradient, $\nabla f$, of $f$, let alone the inverse of its Hessian matrix, $\nabla^2f$, is impractical due to memory constraints. To address such optimization challenges, subspace optimization methods have been extensively investigated. The primary concept underlying most subspace methods is to construct a lower-dimensional problem:

\begin{equation}\label{subspaceprob}
	\min_{\bm{d}\in\mathbb{R}^m} g_k(\bm{d}) := f(\bm{x}_k + P_k\bm{d}).
\end{equation}

Here, $\bm{x}_k\in\mathbb{R}^n$ and $P_k \in\mathbb{R}^{n\times m}$ (with $m \ll n$) are fixed, and the iterate is updated to minimize a first-order, second-order, or higher-order approximation of $f$ concerning $\bm{d}$. Consequently, the gradient and Hessian of $f$ with respect to $\bm{d}$ are computed in $\mathbb{R}^m$. Existing research on subspace algorithms can broadly be categorized into (a) methods for constructing $P_k$, and (b) techniques for approximating $f$ in \eqref{subspaceprob} to derive the search direction $\bm{d}$ while ensuring convergence to a stationary point of \eqref{problem} on average or with high probability.

\paragraph{Existing works}
The issue (a) plays an important role in numerical experimental;
When solving the problem \eqref{subspaceprob}, $\bm{x}_k+P_k\bm{d}$ moves only in an $m$-dimensional subspace of $\mathbb{R}^n$. 
The matrix $P_k$ needs to be chosen so that the function value in \eqref{problem} decreases 
significantly. In this sense, the construction of $P_k$ is essential. However, few studies investigate how to construct $P_k$.
Among the previous studies \cite{Lcomdir,Comdir,kimiaei} that investigate the construction of a deterministic $P_k$, two of them use previous gradients (see \cite{Lcomdir,Comdir}). In these methods, $P_k$ is formed by an orthonormal basis of the span of the previous update directions, and a Newton iteration with line search is applied to \eqref{subspaceprob} at each iteration. Under certain conditions on $P_k$, it is shown that the function value converges to the optimal value for convex functions and functions satisfying the Polyak-Lojasiewicz condition and that the gradient norm converges to zero for nonconvex functions.
Many other methods, such as \cite{R-subspace,Sketch,pilanci2014,Roberts2022-us,gower2019rsn,cubicNewton,CJZ,nesterov2017random} that bear the name of ``random subspace'' methods use random matrices $Q\in\mathbb{R}^{n\times d}$, where each component follows a probability distribution with mean $0$ and variance $1$ and the property 
\begin{equation}\label{randmatrix}
	\mathbb{E}[QQ^{\top}/d]=I_n
\end{equation}
is used for their convergence analysis. 

As for (b), there are various possible algorithms for \eqref{subspaceprob}, just as various nonlinear optimization algorithms exist for \eqref{problem}.
Second-order optimization methods based on second-order approximations of $f$, such as the so-called Newton methods, can easily benefit from the subspace method because the size of the Hessian approximation matrix is reduced to $m \times m$. For this reason, various subspace Newton optimization methods (e.g., \cite{gower2019rsn,cubicNewton,Shao,CJZ,R-subspace,Lcomdir,Comdir}) have been proposed.

For example, \cite{gower2019rsn} proposed randomized subspace Newton for convex problems by computing the search direction using a random matrix $Q_k$ at each iteration where the descent direction $\bm{d}_k$ is given by
$\bm{d}_k = -(Q_k^\top \nabla^2 f(\bm{x}_k) Q_k )^{-1} Q_k^\top \nabla f(\bm{x}_k)$. Restricting the method to subspace greatly reduces the computation of the inverse of the Hessian matrix to 
$O(m^3)$. \cite{cubicNewton} proposed a cubically-regularized subspace Newton method for convex problems. 
Recently, \cite{Shao,CJZ} and \cite{R-subspace} independently proposed random subspace Newton-based methods for nonconvex optimization problems.
The subspace methods of \cite{Lcomdir,Comdir}, where $P_k$ is deterministic, also use Newton iterations for nonconvex problems.
These methods need the Hessian computation in the reduced space $\mathbb{R}^m$.
However, since these methods use the Hessian matrix of the original function $f$, numerical experiments have shown that they are fast when the Hessian matrix can be easily computed, such as when the objective function can be separated for each variable, but in other cases, they are computationally heavy and are not practically applicable to large-scale problems. 
See the detailed discussion on existing works in Section~\ref{previous_work}.
Although existing subspace Newton methods require fewer iterations to converge in practice, the matrix operations and full gradient computation are bottlenecks when dealing with large-scale problems.

\paragraph{Motivation}
When considering the aforementioned lines of research, a natural question arises: Is it possible to construct a random subspace Newton-based method using only subspace gradients for solving large-scale problems with the same order of the worst-case iteration complexities in average as before?
 To answer the research question, we propose a subspace quasi-Newton method for \eqref{problem}, where
 the subspace  $P_k$ in \eqref{subspaceprob} is constructed with
 $m$-dimensional subgradient vectors $Q_k^{\top}\nabla f(\bm{x}_k)$ using a random matrix $Q_k$, an approximate Hessian inverse\footnote{It may be confusing, but by following the conventions of quasi-Newton research (see e.g., \cite{NoceWrig06}), we use $B_k$ for the approximate Hessian matrix of $f$ at $\bm{x}_k$ and $H_k$ for its inverse. Note that neither the Hessian matrix computation nor the inverse matrix calculation is required, and only $H_k$ is updated at each iteration of quasi-Newton algorithms.} $H_k\in \mathbb{R}^{m \times m}$ is constructed by following the standard BFGS formula, and the search direction is computed by
 $\bm{d}_k= -H_k (P_k^{\top}\nabla f(\bm{x}_k)) \in \mathbb{R}^{m}$.
 At the end of each iteration, an $n$-dimensional iterate is updated
 by $\bm{x}_{k+1}= \bm{x}_k+\alpha_kP_k\bm{d}_k$ with a chosen step size $\alpha_k$.

 The unique feature of this method is that it requires only subspace gradient vectors, $P_k^{\top}\nabla f(\bm{x}_k)$ and $Q_k^{\top}\nabla f(\bm{x}_k)$, for constructing the approximate
 Hessian inverse $H_k$ and the subspace matrix construction $P_k$, respectively.
 We can use recent automatic differentiation (AD) techniques for the computation of subspace gradient vectors; indeed, $Q_k^{\top}\nabla f(\bm{x}_k)$ can be expressed as $\nabla \tilde{g}_k(\bm{0})$ using the function $\tilde{g}_k(\bm{d}):=f(\bm{x}_k+Q_k\bm{d})$,  restricted in the subspace.
 Furthermore, the subspace gradient calculation 
can be done by using the finite-difference method.
Even in the presence of errors due to the finite-difference method, our algorithm still maintains, 
in expectation, the same order of the worst-case iteration complexities 
given e.g., by \cite{Lcomdir}. 
In fact, when the gradient is readily available, zeroth-order optimization methods are not competitive with first- or second-order methods:  by using reverse-mode
automatic differentiation (AD), one can theoretically evaluate $\nabla f(x)$ in about four-
times the cost of evaluating $f(x)$, regardless of the dimension $n$: \cite{ADbook}.
However, there are many situations when AD is inefficient or not applicable such as PDE-constrained optimization \cite{kozak2021stochastic}.
When it comes to problems where even function-value evaluation is difficult, 
AD will not work efficiently. Furthermore, even when applicable, the reverse mode of AD may lead to memory issues when non-linear processes appear \cite{ADbook}. In such a case, our algorithm is applicable using either the forward mode AD to compute the gradient component-wise on the subspace, or finite-difference method.

Notice that as Table~\ref{table1} suggests, there has been no algorithm other than the subspace gradient method \cite{kozak2021stochastic,kozak2023} that can be constructed using only the subspace gradient.

\paragraph{Contribution}
We can summarize our contribution as follows.

(1)  We introduce a novel quasi-Newton method for subspace optimization. We compute an approximate inverse Hessian matrix of small size using a matrix $P_k$ that defines a subspace, constructed using the approximate gradient vectors.

(2) Our method maintains convergence guarantees while significantly reducing computational overheads and memory space, making it well-suited for large-scale problems. Our approach achieves computational efficiency without compromising performance by using approximate inverse Hessian matrices and random matrix properties.

This is not the first time that the idea of using the quasi-Newton method in the subspace has been proposed. In fact, 
\cite{R-BFGS} utilizes approximate Hessian matrices instead of the Hessian matrix in their subspace algorithm. However, the update rule for the approximate Hessian needs the true Hessian of the function, which leads to memory problems when the dimension is large.
Compared to their method, our method 
does not require inverse matrix computation nor handle the true Hessian matrix.

\section{Related Work} \label{previous_work}
Our algorithm can be viewed as the stochastic and quasi-Newton variants of the deterministic subspace Newton methods proposed by \cite{Lcomdir}.
Their method uses Hessian matrices and their inverse, while ours does not. Furthermore, our method does not require a full gradient calculation in the original $n$-dimension using a random matrix $Q_k$. In this section, we note the differences between their methods as well as the existing random subspace gradient descent and (quasi-)Newton methods. 

\subsection{Deterministic Subspace Newton Method}\label{Lee}
A subspace Newton's method, using a deterministic subspace, has been studied in \cite{Lcomdir,Comdir,ZhouNewton}. the function $f$ is restricted on a subspace by defining
\begin{equation}
	\min_{\bm{d}\in\mathbb{R}^{m}}g_k(\bm{d}) := f(\bm{x}_k+P_k\bm{d}). \label{subspace}
\end{equation}
\cite{Lcomdir} proposes an efficient Newton update with linear search to \eqref{subspace}. 
The algorithm consists of 
\begin{itemize}
	\item the construction of  $\tilde{B}_k :=P_k^{\top}\nabla^2f(\bm{x}_k)P_k \in \mathbb{R}^{m_k \times m_k}$ (if necessary, modify it so as to satisfy $\sigma_{\text{min}}(\tilde{B}_k)<M$ with some constant $M >0$)
	\item compute $ \bm{x}_{k+1}= \bm{x}_k-\alpha_kP_k \tilde{B}_k^{-1}P_k^{\top}\nabla f(\bm{x}_k)$ with the step-size $\alpha_k$ chosen by the line-search procedure
	\item update $P_{k+1}$ using  $ \bm{x}_{k+1}$ e.g., by \eqref{P_kBFGS} shown later
\end{itemize}
under the following assumptions: 

\eqref{Lipdif}, \eqref{Hessian}, and \eqref{P_k}
for $f$, $H_k\in\mathbb{R}^{m_k \times m_k}$ and $P_k:=[q_1^k,q_2^k,\dots,q_{m_k}^k]\in\mathbb{R}^{n \times m_k}$, respectively.
\begin{equation}\label{Lipdif}
	\exists L>0,\ \forall\bm{u},\bm{v}\in\mathbb{R}^n,\ \|\nabla f(\bm{u})-\nabla f(\bm{v})\|\leq L\|\bm{u}-\bm{v}\|
\end{equation}
\begin{equation}\label{Hessian}
	\exists M_1,M_2\in \mathbb{R}_{>0},\quad \textup{s.t.} \quad M_1I\succeq \nabla^2f(x_k)\succeq M_2I\ (\forall k)
\end{equation}
\begin{equation}\label{P_k}
	\exists \delta\in\mathbb{R}_{>0} \mbox{ and there exists a column vector }  q_j^k \mbox{ of } P_k
	\quad \textup{s.t.}\quad \dfrac{|\nabla f(\bm{x}_k)^{\top}q_j^k|}{\|\nabla f(\bm{x}_k)\|\|q_j^k\|}\geq\delta
\end{equation}
$P_k\in\mathbb{R}^{n \times m_k}$ can be any matrix, as long as it satisfies \eqref{P_k}, including the way to determine the number of columns $m_k$. For example, the following $P_k$ is proposed, where $m_k$ is kept constant equal to $m$.
\begin{equation}\label{P_kBFGS}
	P_k=
	\left\{
	\begin{array}{ll}
		\left[\bm{x}_{k-\frac{m}{2}+1},\nabla f(\bm{x}_{k-\frac{m}{2}+1}),\dots,\bm{x}_k,\nabla f(\bm{x}_k)\right]\ & (k\geq\frac{m}{2}-1) \\
		\left[\bm{x}_0,\nabla f(\bm{x}_0),\dots,\bm{x}_k,\nabla f(\bm{x}_k)\right]\ & (k<\frac{m}{2}-1)
	\end{array}
	\right.
\end{equation}

For the above $P_k$, the expression \eqref{P_k} holds as $\delta=1$, since the column contains $\nabla f(\bm{x}_k)$. 

In \cite{Comdir}, $\{P_k\}$ is updated from  $P_0=[\nabla f(\bm{x}_0)/\|\nabla f(\bm{x}_0)\|]$
by following
\begin{equation}\label{P_kComdir}
	P_{k+1}=
	\left\{
	\begin{array}{ll}
		\left[P_k,\bm{p}_k/\|\bm{p}_k\|\right]\ & (\bm{p}_k\neq\bm{0})\\
		P_k\ & (\bm{p}_k=\bm{0}),
	\end{array}
	\right.
\end{equation}
where $\bm{p}_k=\nabla f(\bm{x}_{k+1})-P_kP_k^{\top}\nabla f(\bm{x}_{k+1})$.


The convergence properties of \cite{Lcomdir,Comdir}'s methods are shown for the cases where the objective function $f$ is nonconvex, convex, and satisfies the Polyak-Lojasiewicz condition, respectively. In \cite{ZhouNewton}, the authors use a discrete subspace, where only a subset of the coordinate vector $x$ is considered. At each step, the coordinates that need to be considered in the subspace are updated according to some rule. In this method, the size of the subspace is not fixed.

\subsection{Random Subspace Newton-based Methods}
Second-order optimization methods based on second-order approximations of $f$, such as the so-called Newton methods, can easily benefit from the subspace method because the size of the Hessian matrix is reduced to $m \times m$. For this reason, various subspace Newton optimization methods have been proposed.
For example, \cite{gower2019rsn} proposed a randomized subspace Newton for convex problems by computing the search direction using a random matrix $Q_k \in \mathbb{R}^{n \times m}$ at each iteration as
$d_k = -(Q_k^\top  \nabla^2 f(x_k) Q_k )^{-1} Q_k^\top \nabla f(x_k)$. Restricting the method to a subspace greatly reduces the computation of the inverse of the Hessian matrix to 
$O(m^3)$. \cite{cubicNewton} proposed a cubically-regularized subspace Newton method for convex problems. 
Recently, \cite{Shao,CJZ}, \cite{R-subspace}, \cite{GNC} uses second-order subspace methods to solve the non-convex optimization problem.
The subspace methods of \cite{Lcomdir,Comdir} with deterministic construction of $P_k$ also use Newton iteration to solve a nonconvex problem. These methods only need to compute the Hessian in the projected space $\mathbb{R}^m$.
However, since these methods use the Hessian matrix of the original function $f$, numerical experiments have shown that they are fast when the Hessian matrix can be easily computed, such as when the objective function can be separated for each variable, but in other cases, they are computationally heavy and are not practically applicable to large-scale problems. In other cases, however, it is expected to be computationally expensive and not practically applicable to large-scale problems.

\subsection{Existing Random Subspace Quasi-Newton Method}\label{Kovalev}
The random subspace quasi-Newton method of 
\cite{R-BFGS} replaces the Hessian matrix of $f$ with a matrix whose inverse $H_{k}$ is easy to compute and
updates iterates as $\bm{x}_{k+1}=\bm{x}_k-H_k\nabla f(\bm{x}_k)$. The sequence of iterates is shown to satisfy $\sqrt{f(\bm{x}_k)-f^*}$ converges locally superlinearly to 0 with probability 1. 

In the paper, the inverse matrix of a random subspace approximate Hessian is described as 
\begin{equation}
	H_{k+1}=\tilde{B}_k+(I-\tilde{B}_k \nabla^2f(\bm{x}_k))H_{k}(I-\nabla^2f(\bm{x}_k)\tilde{B}_k),
	\label{BFGS_existing}
\end{equation}
where $\tilde{B}_k:=Q_k(Q_k^{\top}\nabla^2f(\bm{x}_k)Q_k)^{-1}Q_k^{\top}$ using a random matrix $Q_k\in\mathbb{R}^{n\times m}(m\ll n)$.
The method can avoid taking the inverse of approximate Hessian matrices in the full-dimensional space but still needs the matrix product in the full dimension.

\subsection{Random Subspace Gradient-based Methods}
\cite{kozak2021stochastic} and \cite{kozak2023} propose a random subspace gradient-based methods. At each step, the gradient is approximately computed on the random subspace, using zeroth order approximation. By doing so, we can decrease the computation of the gradient, when automatic differentiation cannot be used, from $O(n^2)$ to $O(nd)$. This approach is especially efficient when there is no closed-form expression for the gradient or when it is costly to evaluate.

\section{Proposed Method}\label{research}

\begin{table}
	\caption{Subspace methods for unconstrained optimization. $\doublecheck$ in the column ``Subspace gradient'' indicates that approximate subspace gradient is computed by the finite difference method, and $\doublecheck$ in the column ``Subspace Hessian'' indicates that approximate subspace Hessian is computed without computing the Hessian matrix) }
	\label{table1}
	\centering
	\begin{tabular}{lccccc}
		\toprule
		& Subspace & Subspace  & Matrix  & \multicolumn{2}{c}{Type of Subspace} \\
                \cline{5-6}
		& gradient & Hessian & inv. free & Random $Q_k$ & Deterministic $P_k$\\
		\midrule 
		 \cite{kozak2021stochastic} &  $\checkmark$ & &$\checkmark$ & $\checkmark$& \\
		\midrule 
		 \cite{kozak2023} &  $\doublecheck$ & &$\checkmark$ & $\checkmark$& \\                
		\midrule 
		 \cite{gower2019rsn,R-BFGS,Shao,CJZ,GNC}  &  &$\checkmark$ & &$\checkmark$ & \\
		\midrule 
		 \cite{R-BFGS} &  & $\checkmark$& &$\checkmark$ & \\
		\midrule 
		\cite{Lcomdir} &  & $\checkmark$ & & &$\checkmark$ \\
		\midrule 
		\makecell{Ours:exact} & $\checkmark$ & $\doublecheck$& $\checkmark$&$\checkmark$ & $\checkmark$ \\
		\midrule 
	 \makecell{Ours:approx.} & $\doublecheck$ &$\doublecheck$ & $\checkmark$&$\checkmark$ & $\checkmark$\\
		\bottomrule
	\end{tabular}
\end{table}

\subsection{Randomly Projected Gradient and Construction of $P_k$}\label{method}
This section presents a construction of $P_k$ based on \cite{Lcomdir}. 
The results of \cite{Lcomdir} suggest that we could obtain benefits by incorporating gradient information into $P_k$. However, directly computing the gradient of $f$ often requires $O(n^2)$ operations, especially if we cannot use automatic differentiation techniques. Let us consider $Q_k \in \mathbb{R}^{n \times d}$, where $d \ll n$, such that each element follows a standard normal distribution.
Noticing \eqref{randmatrix} 
and using the concentration of $Q_kQ_k^{\top}/d$ around it's expectation (see \cite{random}), we have
\begin{equation}\label{approx_grad}
	\nabla f(\bm{x}_k)\approx Q_k (Q_k^{\top}\nabla f(\bm{x}_k))/d.
	\notag
\end{equation}
Consider 
the restriction of $f$ on the $d$-dimensional subspace as 
\begin{equation}
	\tilde{g}_k(\bm{d})=f(\bm{x}_k+Q_k\bm{d}),
	\notag
\end{equation}
and  notice that $\nabla \tilde{g}(\bm{0})=Q_k^{\top}\nabla f(\bm{x}_k)$.
To compute the gradient of this function, let $\bm{d}_i$ denote the $i$-th component of $\bm{d}$ and 
let $\bm{e}_i$ denote the $i$th element of the canonical basis. We have
\begin{equation}\label{subgrad}
	\dfrac{\partial \tilde{g}_k(0)}{\partial \bm{d}_i} = \lim_{\varepsilon \to 0} \dfrac{f({\bm{x}_k}+\varepsilon Q_k\bm{e}_i)-f({\bm{x}_k}-\varepsilon Q_k\bm{e}_i)}{2\varepsilon}, ~~ i=1,\ldots, d.
\end{equation}
This means that we only need $O(d)$ evaluations of $f$ to estimate $Q_k^{\top}\nabla f(\bm{x}_k)$. 

We modify $P_k$ of \cite{Lcomdir} (see the original definition \eqref{P_kBFGS}  by \cite{Lcomdir}) to use subspace gradient vectors
generated with a random matrix $Q_k$ at iteration $k$. More specifically we compute the matrix $P_k \in \mathbb{R}^{n \times m}$ as
\begin{equation}\label{approxP_k}
	P_k=\left[\dfrac{\bm{x}_{k-\frac{m}{2}+1}}{\|\bm{x}_{k-\frac{m}{2}+1}\|},\dfrac{Q_{k-\frac{m}{2}+1}Q_{k-\frac{m}{2}+1}^{\top}\nabla f(\bm{x}_{k-\frac{m}{2}+1})}{\|Q_{k-\frac{m}{2}+1}Q_{k-\frac{m}{2}+1}^{\top}\nabla f(\bm{x}_{k-\frac{m}{2}+1})\|},\dots,\dfrac{\bm{x}_k}{\|\bm{x}_k\|},\dfrac{Q_kQ_k^{\top}\nabla f(\bm{x}_k)}{\|Q_kQ_k^{\top}\nabla f(\bm{x}_k)\|}\right],
\end{equation}
where each component having a zero norm, in the denominator, is replaced by $0$.
Unlike the method proposed in \cite{Lcomdir}, the number of columns of $P_k$ is constant along the iterations. We appropriately define the initial matrix $P_0\in\mathbb{R}^{n\times m}$ under the condition that each column is normalized. At each iteration the two first columns of $P_k$ are deleted and the two columns $\dfrac{\bm{x}_k}{\|\bm{x}_k\|},\dfrac{\nabla f(\bm{x}_k)}{\|\nabla f(\bm{x}_k)\|}$ are added in $P_{k+1}$. Similarly to $\tilde{g}$, we define:
\begin{equation}\label{eq:functgk}
	{g}_k(\bm{d})=f(\bm{x}_k+P_k\bm{d}).
	\notag
\end{equation}
We also have
\begin{equation}\label{subgrad1}
	P_k^{\top}\nabla f(\bm{x}_k)=\dfrac{\partial {g}_k(0)}{\partial \bm{d}_i} = \lim_{\varepsilon \to 0} \dfrac{f({\bm{x}_k}+\varepsilon P_k\bm{e}_i)-f({\bm{x}_k}-\varepsilon P_k\bm{e}_i)}{2\varepsilon}, ~~ i=1,\ldots, m.
\end{equation}
\subsection{Approximate Hessian Inverse Matrix}\label{prepare}

For the approximate Hessian matrix $B_k^{F}$ of full-dimensional BFGS at the $k$th iteration,
the inverse matrix, $H_{k}^{F}$, of $B_k^{F}$  
is updated as
\begin{equation}\label{BFGS}
	\mbox{BFGS formula}: H_{k+1}^{F}=\textup{BFGS}(H_k^{F},\bm{s}_k,\bm{y}_k)=\left(I-\dfrac{\bm{s}_k\bm{y}_k^{\top}}{\bm{s}_k^{\top}\bm{y}_k}\right)H_k^{F}\left(I-\dfrac{\bm{y}_k\bm{s}_k^{\top}}{\bm{s}_k^{\top}\bm{y}_k}\right)+\dfrac{\bm{s}_k\bm{s}_k^{\top}}{\bm{s}_k^{\top}\bm{y}_k},
\end{equation}
where $\bm{s}_k:=\bm{x}_{k+1}-\bm{x}_k,\quad \bm{y}_k:=\nabla f(\bm{x}_{k+1})-\nabla f(\bm{x}_k).$
Instead, we apply the BFGS formula to the function $g_k$. $\bm{s}_k,\bm{y}_k$ now become
\begin{equation}\label{subQuasi}
	\bm{s'}_k=P_k^\top(\bm{x}_{k+1}-\bm{x}_k),\qquad \bm{y'}_k=\nabla {g_k}(\bm{d}_k)-\nabla {g_k}(\bm{0})=P_k^{\top}(\nabla f(\bm{x}_{k+1})-\nabla f(\bm{x}_k)).
\end{equation}
Here, $P_k\in\mathbb{R}^{n\times m}$, and $\bm{s'}_k,\bm{y'}_k\in\mathbb{R}^{m},H_k,H_{k+1}\in\mathbb{R}^{m \times m}$. 
The equation below update $H_{k}$ along a lower-dimensional subspace represented by $P_k$ as 
\begin{equation}\label{subBFGS}
	\tilde{H}_{k+1}=\textup{BFGS}(H_k,\bm{s'}_k,\bm{y'}_k).
\end{equation}

In the algorithm, the matrix $\tilde{H}_{k}$  is modified to a matrix $H_k$ such that 
\begin{equation}\label{B_knorm}
	M_1I_m\preceq H_{k}\preceq M_2I_m\ (\forall k)
\end{equation}
is satisfied. Here $M_1$ and $M_2$ are pre-chosen parameters.
To make the approximate Hessian inverse matrix positive definite, it is enough to check
if ${\bm{s'}_k}^\top \bm{y'}_k > 0$, as confirmed by the update rule \eqref{BFGS}. 

The requirement of \eqref{B_knorm} is for the worst-case iteration complexity analysis.
For a long time, guarantees of global convergence for quasi-Newton have been primarily limited to asymptotic ones (e.g., \cite{powell71,byrd96}).
There have been recent studies  (see e.g., \cite{nakayama21COAP,liu22COAP}) that show iteration complexity analysis for quasi-Newton methods, but all of them have modified the matrix $H_k$ or assumed the positive definiteness for $H_k$.
Although quite recently, 
iteration complexity analysis has been derived for the quasi-Newton method using the formula \eqref{BFGS} for the first time by \cite{jin2024globalconvBFGS} without any processing for $H_k$, we modify the matrix $H_{k}$ for satisfying \eqref{B_knorm} in this paper.

There are various ways for modifying a matrix to satisfy it, but since the modification for $m \times m$ matrices is not computationally expensive, we will modify it using the eigenvalue computation as follows:
set eigenvalues of $\tilde{H}_{k+1}$ larger than $M_2$ or smaller than $M_1$ to $M_2$ and $M_1$, respectively. 
Let $\textup{Modify\_eig}(\tilde{H}_{k+1})$ denote this modification. Notice that we only need to store an $m \times m$ matrix here. Furthermore, $\textup{Modify\_eig}$ can be done in $O(m^3)$ which is negligible compared to other operations such as the update of $\bm{x}_k$ that needs $O(nm)$ operations.

\subsection{Algorithm}
Based on the above discussion, the algorithm of the proposed method is detailed in Algorithm~\ref{alg7}.

\begin{algorithm}[tb]
	\caption{Subspace Quasi-Newton Method with Randomly Approximated Gradient}
	\label{alg7}
	Input: $\bm{x}_0$,  $P_0,\ M_1I_m\preceq H_0\preceq M_2I_m$, $\beta,c\in(0,1)$
	
	parameters: $M_1,M_2>0$.
	\begin{algorithmic}[1]
		\For {$k=0,1,\dots$}
		\State Sample  $(Q_k)_{ij}\ \sim\ N(0,1)\qquad(i=1,2,\dots,n\quad j=1,2,\dots,d)$.
		
		\State Compute $\nabla \tilde{g}_k(\bm{0})=Q_k^\top \nabla f(\bm{x}_k)$ using \eqref{subgrad}.
		\State Set $P_{k}$ by using \eqref{approxP_k} with $\bm{x}_{k}$ and $Q_k (Q_k^{\top}\nabla f(\bm{x}_k))$, and construct $g_k$ as in \eqref{subspaceprob}.
		\State $\bm{d}_k\ \leftarrow\ -H_k\nabla g_k(\bm{0})$ \label{eq:algodescent}
		\For{$i=0,1,\dots$}  
		\label{alg:linesearch_s}
		\State $\alpha_k\ \leftarrow\ \beta^i$
		\If {$f(\bm{x}_k+\alpha_kP_k\bm{d}_k)\leq g_k(\bm{0})+c\alpha_k\nabla g_k(\bm{0})^{\top}\bm{d}_k$ \label{eq:algoarmijo}} 
		\State break
		\EndIf
		\EndFor            
		\label{alg:linesearch_e}     
		\State $\bm{x}_{k+1}\ \leftarrow\   \bm{x}_k+\alpha_kP_k\bm{d}_k$
		\State $\tilde{H}_{k+1}\ \leftarrow\ \textup{BFGS}(H_k,\alpha_k\bm{d}_k,\nabla g_k(\alpha_k\bm{d}_k)-\nabla g_k(\bm{0}))$
		\label{alg:invmat_s}
		\State $H_{k+1}\ \leftarrow\ $Modify\_eig$(\tilde{H}_{k+1})$
		\label{alg:invmat_e}
		\EndFor 
	\end{algorithmic}
\end{algorithm}

The procedure from \cref{alg:linesearch_s} to \cref{alg:linesearch_e} is for 
the linear search determining the step-size $\alpha_k$ so as to satisfy Armijo's condition for the function $g_k$:
\begin{equation}\label{linsearch}
	{g_k}(\bm{d}_k)\leq {g_k}(\bm{0})+c\alpha_k\nabla {g_k}(\bm{0})^{\top}\bm{d}_k.
\end{equation}
At first, we will show a lower bound, for all $k$,
for the step-size $\alpha_k$ computed by the line-search procedure.
\begin{theorem}\label{stepsize2}
	Suppose that Assumption~\ref{Lipdif2} is satisfied. For any $k$, the step size $\alpha_k$ satisfies 
	\begin{equation}\label{stepsize3}
		\alpha_k\geq\bar{\alpha}:=\min\left(1,\dfrac{2\beta(1-c)}{mM_2L}\right).
	\end{equation}
\end{theorem}
All of the proofs in this paper are in Appendix~\ref{sec:allproof}, including the proof of this theorem (see  Appendix~\ref{sec:stepsize_proof}).

Theorem~\ref{stepsize2} shows that $\alpha_k$ is lower bounded by a constant $\bar{\alpha}$, which implies that
the iteration number from \cref{alg:linesearch_s} to \cref{alg:linesearch_e} is bounded 
and does not affect the worst-case iteration complexity.
The procedure \cref{alg:invmat_e} is to ensure \eqref{B_knorm}.
To make the algorithm more practical, we can replace it with a simpler rule such as applying \cref{alg:invmat_e} only if ${\bm{s'}_k}^\top \bm{y'}_k < 0$.

It should be noted that the algorithm does not need the calculation of the full-dimensional vector $\nabla f(\bm{x}_k)$ nor the (subspace) Hessian matrix. We, however, need to estimate $Q_k^\top\nabla f(\bm{x}_k)$ as in \eqref{subgrad}.
The main computation cost per iteration comes from the evaluation of the following subspace gradients: $\nabla \tilde{g}_k(\bm{0})=Q_k^\top \nabla f(\bm{x}_k)$, $\nabla g_k(\bm{0})=P_k^\top \nabla f(\bm{x}_k)$ and $\nabla g_k(\alpha_k\bm{d}_k)$.

\section{Global Convergence Rate}\label{subsec:1}
Before starting the convergence theorems, we recall two lemmas on random matrices:
\begin{lem}\label{rand_norm2}
	\textup{(\cite[Lemma 5.3.2]{random},\cite[Lemma 2.1]{R-subspace})}
	Let $Q\in\mathbb{R}^{n\times d}$ be a random Gaussian matrix whose elements follow independently a $N(0,1)$ distribution. There exists a constant $\mathcal{C}_0$ such that for any $0<\epsilon<1$ and for any $\bm{x}\in\mathbb{R}^n$,
	\begin{equation}\label{vectornorm}
		(1-\epsilon)\|\bm{x}\|^2\leq\left\|\dfrac{1}{\sqrt{d}}Q^{\top}\bm{x}\right\|^2\leq(1+\epsilon)\|\bm{x}\|^2
	\end{equation}
	with probability at least $1-2e^{-\mathcal{C}_0\epsilon^2d}$. 
\end{lem}

\begin{lem}\label{rand_norm}	
	\textup{(\cite[Lemma 2.2]{R-subspace})}
	Let $Q\in\mathbb{R}^{n\times d}$ be a random matrix whose elements follow independently a $N(0,1)$ distribution. There exists a constant $\mathcal{C}$ such that
	\begin{equation}\label{oparationnorm}
		\dfrac{1}{d}\|QQ^{\top}\|_{\textup{op}}\leq \mathcal{C}\dfrac{n}{d}
	\end{equation}
	with probability at least $1-2e^{-d}. $
\end{lem}

Let us fix $\epsilon\in (0,1)$. From Lemmas~\ref{rand_norm2} and \ref{rand_norm}, we have that both \eqref{vectornorm} and \eqref{oparationnorm} hold with
probability at least $1-(2e^{-\mathcal{C}_0\epsilon^2d}+2e^{-d})$ for a random matrix $Q_k$.
From then on let us define
\begin{equation}\label{peps}
	p(\epsilon):=1-(2e^{-\mathcal{C}_0\epsilon^2d}+2e^{-d}).
\end{equation}
The following lemma is a direct consequence of Lemma \ref{rand_norm}.

\begin{lem}\label{P_knorm}
	The operator norm of $P_k$ satisfies
	\begin{equation}
		\|P_k\|_{\textup{op}}=\sup_{\|\bm{v}\|\neq0}\dfrac{\|P_k\bm{v}\|}{\|\bm{v}\|}=\sup_{\|\bm{v}\|=1}\|P_k\bm{v}\|\leq \sqrt{m}.
		\notag
	\end{equation}
\end{lem}

Let us consider the following assumption: 
\begin{asp}\label{Lipdif2}
	$\exists L>0, 
	\|\nabla f(\bm{u})-\nabla f(\bm{v})\|\leq L\|\bm{u}-\bm{v}\|$
	holds for $\forall\bm{u},\bm{v}\in\mathbb{R}^n$.
\end{asp}

Under this assumption, we derive, in the following two subsections, the worst-case iteration complexity for our algorithm using either automatic differentiation or finite difference based on \eqref{subgrad}.

\subsection{Exact Subspace Gradient Computation } \label{sec:exact_subgrad}
In this subsection, we analyze the global convergence speed of our algorithm assuming that the gradients $Q_k^{\top}\nabla f(\bm{x}_k)$ and $P_k^{\top}\nabla f(\bm{x}_k)$ are computed exactly e.g., by automatic differentiation.
In the following, we present three convergence theorems, respectively, when $f$ is nonconvex, convex, and satisfies the Polyak-Lojasiewicz condition.

\begin{theorem}\label{nonconvex2}
	Suppose that Assumption \ref{Lipdif2} is satisfied. Then for any $\epsilon \in (0,1)$ and any initial point initial point $\bm{x}_0$, $ \mathbb{E}\left[\|\nabla f(\bm{x}_j)\|\right]$ converges to $0$ and 
	\begin{equation}\label{eq:nonconvex2}
		\mathbb{E}\left[\min_{0\leq j\leq k}\|\nabla f(\bm{x}_j)\|\right]\le \sqrt{\frac{1}{k+1} \left(c\bar{\alpha} M_1\left(\dfrac{(1-\epsilon)d}{\mathcal{C}n }\right)^2p(\epsilon)\right)^{-1}(f(\bm{x}_0)-f^*)} .
	\end{equation} 
\end{theorem}
Here $\bar{\alpha}$ is defined 
in \eqref{stepsize3}, $c, M_1$ are parameters of Algorithm~\ref{alg7}, $\mathcal{C}$ and $p(\epsilon)$ are from \eqref{oparationnorm} and \eqref{peps}, respectively, and $f^*$ is the optimal value of the concerned problem.

\begin{theorem}\label{convex2}
	Suppose that Assumption~\ref{Lipdif2} is satisfied. In addition, suppose that $f$ is convex and for any initial point $\bm{x}_0$, let us assume that 
	\begin{equation}\label{R_0}
		R_0=\max_{\bm{x}:f(\bm{x})\leq f(\bm{x}_0)}\min_{\bm{x}^* \in \Omega}\|\bm{x}-\bm{x}^*\|
	\end{equation}
	is finite (Here $\Omega$ denotes the set of optimal solutions). Then for all  $\epsilon \in (0,1)$ and for all $k\ge 0$, 
	\begin{equation}
		\mathbb{E}[f(\bm{x}_k)-f^*] 
		\leq\dfrac{L \mathcal{C}^2n^2R_0^2}{2\mathcal{C}^2n^2+L cM_1p(\epsilon)(1-\epsilon)^2d^2k\bar{\alpha}}.
		\notag
	\end{equation}
\end{theorem}
The above-bound parameters are the same as those in Theorem~\ref{nonconvex2}.

\begin{theorem}\label{PL2}
	Suppose that Assumption \ref{Lipdif2} is satisfied and $f$ satisfies the \textup{Polyak-Lojasiewicz condition}  i.e.
	\begin{equation}\label{PL}
		\forall\bm{x} \in \mathbb{R}^n,\qquad\|\nabla f(\bm{x})\|^2\geq2\sigma(f(\bm{x})-f^*)
	\end{equation}
        for the optimal value $f^*$.
	Then for any initial point $\bm{x}_0$, for all  $\epsilon \in (0,1)$ and for all $k\ge 0$, 
	\begin{equation}
		\begin{split}
			\mathbb{E}[f(\bm{x}_{k})-f^*] 
			&\leq\left(1-2\sigma c\bar{\alpha} M_1p(\epsilon)\left(\dfrac{(1-\epsilon)d}{\mathcal{C}n}\right)^2\right)^k(f(\bm{x}_0)-f^*).
		\end{split}
		\notag
	\end{equation}
\end{theorem}

From the above analysis, we can confirm that our algorithm still maintains, in expectation, the same order of the worst-case iteration complexities given, e.g., by \cite{Lcomdir},
though the computational complexity of our algorithm per iteration is reduced, thanks to the computation on the subspace. 

\subsection{Subspace Gradient Approximation by Finite Difference}\label{subsect:2}

Let us consider the modification to Algorithm \ref{alg7} assuming that we use finite difference to compute the gradients. As discussed in Introduction, AD may not work efficiently
for problems where even function-value evaluation is difficult.
First, we describe the changes to Algorithm \ref{alg7}, and then reprove the results of Section \ref{sec:exact_subgrad}. 
Let $\varepsilon>0$ be a fixed precision and let us denote by ${\nabla}_{Q,\varepsilon}f(\bm{x})$ the approximation of the gradient of $\tilde{g}_k(\bm{d})$, that is
\begin{equation}\label{finiteDiff_Q}
\forall i,\ 1\le i\le d,\quad {\nabla}_{Q,\varepsilon}f(\bm{x})_i= \dfrac{f(\bm{x}+\varepsilon Q\bm{e}_i)-f(\bm{x}-\varepsilon Q\bm{e}_i)}{2\varepsilon}.
\end{equation}
Notice that for all $\bm{x} \in \mathbb{R}^n$
\[\lim\limits_{\varepsilon \to 0} \nabla_{Q,\varepsilon}f(\bm{x})=Q^\top \nabla f(\bm{x}).\]
Similarly, we also define
\begin{equation}\label{finiteDiff_P}
\forall i,\ 1\le i\le m,\quad {\nabla}_{P,\varepsilon}f(\bm{x})_i= \dfrac{f(\bm{x}+\varepsilon P\bm{e}_i)-f(\bm{x}-\varepsilon P\bm{e}_i)}{2\varepsilon}.
\end{equation}

The descent direction $\bm{d}_k$, in line \ref{eq:algodescent} of Algorithm \ref{alg7}, is now replaced by
\begin{equation}\label{eq:descent2}
	\bm{d}_k=-H_k {\nabla}_{P_k,\varepsilon}f(\bm{x}_k),
\end{equation}
and the Armijo condition in line \ref{eq:algoarmijo} is now replaced by:
\begin{equation}\label{eq:armijo2}
	g(\bm{x}_k+\alpha_kP_k\bm{d}_k)\leq g_k(\bm{0})+c\alpha_k{\nabla}_{P_k,\varepsilon}f(\bm{x}_k)^{\top}\bm{d}_k.
\end{equation}
Furthermore the columns of $P_k$ in \eqref{approxP_k} now contain
\[\dfrac{Q_k{\nabla}_{Q_k,\varepsilon}f(\bm{x}_k)}{\|Q_k{\nabla}_{Q_k,\varepsilon}f(\bm{x}_k)\|}.\]
Then, the resulting algorithm is the modified Algorithm \ref{alg7} when the gradients are computed by finite difference (see Algorithm~\ref{alg8} for a summary of the modified Algorithm).

Let us consider the computational complexity of the modified Algorithm \ref{alg7}.
The gradients of the function in the subspaces can be computed in $O((d+m)n)$ by using
\eqref{finiteDiff_Q} and \eqref{finiteDiff_P}. 
The update of $H_k$, in \eqref{subBFGS}, needs $O(m^2)$ computations, 
and the eigenvalue calculation is in the order of $O(m^3)$. 
The current iterate $\bm{x}_k$ can be updated in $O(Amn)$ operations, where $A$ is the number of loops needed to satisfy Armijo's condition. From the above, we deduce that if $m \ll n$, modified Algorithm \ref{alg7} needs a total of $O(A(m+d)n)$. From Theorem \ref{stepsize2}, $A$ is in the order of $O(\log_{\beta}\bar{\alpha})$, that is, $O(\log m)$. Hence, the total complexity is $O(n(m+d)\log m)$.

Now we will show the following theorems for the modified Algorithm \ref{alg7} using approximate subspace gradient.
\begin{theorem}\label{nonconvex2b}
	Assume that $c<1/2$ and that Assumption \ref{Lipdif2} is satisfied. Then for any $\epsilon \in (0,1)$ and for any $\varepsilon>0$ and for any initial point $\bm{x}_0$, we have
	\begin{equation}\label{eq:nonconvex2b}
		\mathbb{E}\left[\min_{0\leq j\leq k}\|\nabla f(\bm{x}_j)\|\right] \le \sqrt{\dfrac{1}{k+1} \left(\frac{c\bar{\alpha} M_1 p(\epsilon)}{4} \left(\dfrac{(1-\epsilon)d}{3\mathcal{C}n}\right)^2 \right)^{-1}(f(\bm{x}_0)-f^*) + K^2\varepsilon^2},
	\end{equation}
	where $K:= \max\left(\dfrac{ 4 \sqrt{d}\mathcal{C}n}{1-\epsilon}, \dfrac{6\mathcal{C}nm^2}{d} \left(\dfrac{M_2}{M_1 c}+1 \right) \right)L$, and  
$\bar{\alpha}=\min\left(1,\dfrac{2\beta(1-2c)}{mM_2L}\right)$.
	If we decrease $\varepsilon$ at each iteration so that $\sum \varepsilon^2_k < \infty$, then we have that
	\[\mathbb{E}\left[\min_{0\leq j\leq k}\|\nabla f(\bm{x}_j)\|\right] \le \sqrt{\dfrac{1}{k+1} \left(\frac{c\bar{\alpha} M_1 p(\epsilon)}{4} \left(\dfrac{(1-\epsilon)d}{3\mathcal{C}n}\right)^2 \right)^{-1}(f(\bm{x}_0)-f^*) + \frac{K^2\sum\varepsilon_k^2}{k+1}}.\]
	Hence we deduce that in this case, $\mathbb{E}\left[\min_{0\leq j\leq k}\|\nabla f(\bm{x}_j)\|\right] =O\left(\sqrt{\frac{1}{k}}\right)$.
\end{theorem}

\begin{theorem}\label{convex3}
	Assume that $c<1/2$ and that Assumption \ref{Lipdif2} is satisfied. Assume furthermore that $f$ is convex. and that $R_0$, defined as in Theorem \ref{convex2}, is finite.	Then, for any $\epsilon \in (0,1)$ and for any $\varepsilon>0$ and any initial point $\bm{x}_0$, we have
	\begin{equation}
		\mathbb{E}[f(\bm{x}_k)-f^*] 
		\leq\dfrac{L \mathcal{C}^2n^2R_0^2}{2\mathcal{C}^2n^2+L cM_1p(\epsilon)(1-\epsilon)^2d^2k\bar{\alpha}}+2R_0K\varepsilon.
		\notag
	\end{equation}
	If furthermore we decrease $\varepsilon$ at each iteration $k$ in such a way that $\varepsilon_k\le O(\frac{1}{k})$, then we deduce that $\mathbb{E}[f(\bm{x}_k)-f^*]=O\left(\frac{1}{k}\right)$. 
\end{theorem}

\section{Numerical Experiment}\label{experiment}
We compare the proposed method numerically with other methods for different values of $m$ and $d$. 
All experiments were done with Python 3.10 on an Intel(R) Xeon(R) CPU E5-2695 v4 @ 2.10GHz with a Tesla V100-SXM2-16GB GPU and 503GB of RAM.

For numerical reasons, a modification is made to the proposed method;
Notice that the update in \eqref{subBFGS} requires that $\bm{s}_k^{\top}\bm{y}_k>0$, but to avoid numerical instability when $\|\bm{s}_k\|$ or $\|\bm{y}_k\|$ is too small, 
we modify the proposed method by setting $H_{k+1}=I_m$ when  $\bm{s}_k^{\top}\bm{y}_k< \epsilon$, which does not affect Theorem~\ref{method}. 
The initial matrix $P_0 \in \mathbb{R}^{n \times m}$ of Algorithm \ref{alg7} is build as follows: the first $m-2$ columns are the canonical basis vectors $\{\bm{e}_i\}_{1\le i\le m-2}$ of $\mathbb{R}^n$. The last two columns are the following vectors $\dfrac{\bm{x}_0}{\|\bm{x}_0\|}$ and $\dfrac{Q_0Q_0^{\top} \nabla f(\bm{x}_0)}{\|Q_0Q_0^{\top}\nabla f(\bm{x}_0)\|}$, where $Q_0$ is a random Gaussian matrix.

In all the experiments the precision $\varepsilon$ used to approximate the gradients is fixed to $\varepsilon=0.0001$. 
In addition, for the existing second order method~\cite{R-subspace,Comdir}, we computed the subspace Hessian matrix $P_k^\top \nabla ^2f(x_k)P_k$ using Hessian-vector product, i.e.,
$(P_k^\top \nabla^2 f(x)P_k)_{ij} = p_i^\top \mathrm{hvp}(x,p_j,f)$,
where $\mathrm{hvp}(x,p,f):= \nabla^2 f(x)p$. Notice that only $m$ Hessian-vector products are needed to compute the subspace Hessian.
We computed both directional derivatives and Hessian-vector products using automatic differentiation (jax package in Python). 
Let us consider the following problem:
\paragraph{$l_2$-regularized neural network}
We use the MNIST datasets, \cite{mnist}.
This problem is formulated as 
\begin{equation}
	\min_{\bm{w} \in \mathbb{R}^n}\frac{1}{N}\sum_{i=1}^N\mathcal{L}(\mathcal{M}(w,x_i),y_i)+\lambda\|\bm{w}\|^2,
	\notag
\end{equation}
where $(x_i,y_i)$ denotes the MNIST dataset ($N = 30000$), $\mathcal{L}$ denotes the Cross Entropy Loss function, and $\mathcal{M}$ denotes either 1) the three-layer linear neural network (LNN) with $n=669706$ parameters or 2) convolutional neural network (CNN) with $n=33738$ parameters. In the LNN case, we use Algorithm \ref{alg7} with exact gradient. In the CNN case, we use both exact gradient computation and finite difference gradient (Algorithm \ref{alg8}) for our method. Here, we present the results of the finite difference case. The exact case is presented in Figure \ref{fig:realworld_data_ad} in Appendix~\ref{sec:add_numerics}.
In Section \ref{proposeonly}, we investigate the effect of the dimensions $d$ and $m$ on the performances of our algorithm using the LNN model. In Section \ref{comparepeform}, we compare our method to state-of-the-art methods using the CNN. 
We set $\lambda = 10^{-4}$ in this numerical experiment.


\subsection{Comparison between proposed methods}\label{proposeonly}
We present the results of comparing the proposed method, Algorithm \ref{alg7}, for different values of $m$.

The parameters of Algorithm \ref{alg7}  used in it are defined as follows: $	M_1=0.01,~M_2=1000,~\beta=0.8,c=0.3~$. 
The numerical results under these conditions are shown in Figure~\ref{fig:compare proposed method parameters}.

\begin{figure}[tb]
	\begin{minipage}{0.5\hsize}  
		\centering
		\includegraphics[width=2.8 in]{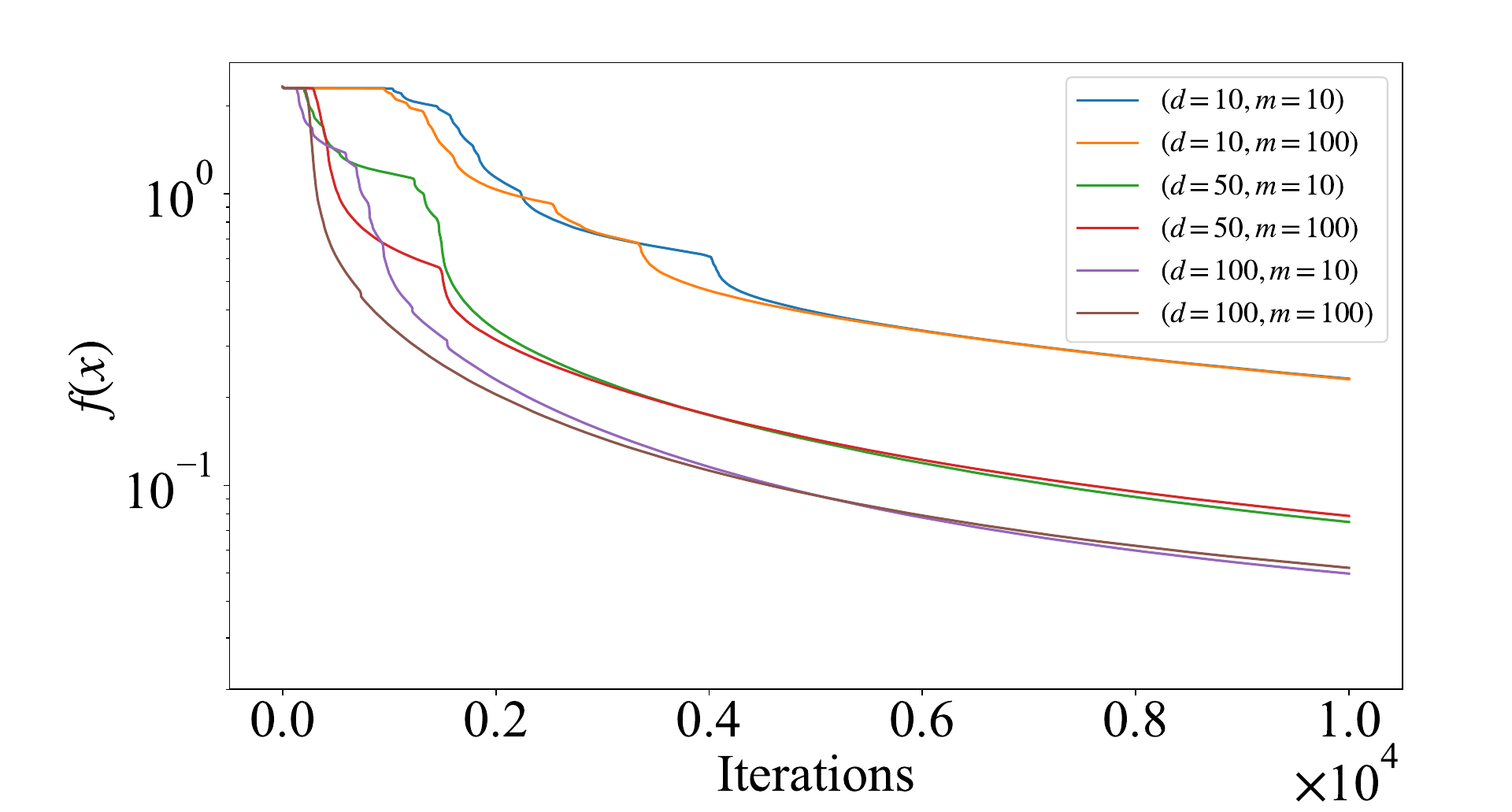}
		\subcaption{Iterations vs. function value}
		\label{fig:compare proposed method parameters(iteration)}
	\end{minipage}
	\begin{minipage}{0.5\hsize}  
		\centering
		\includegraphics[width=2.8 in]{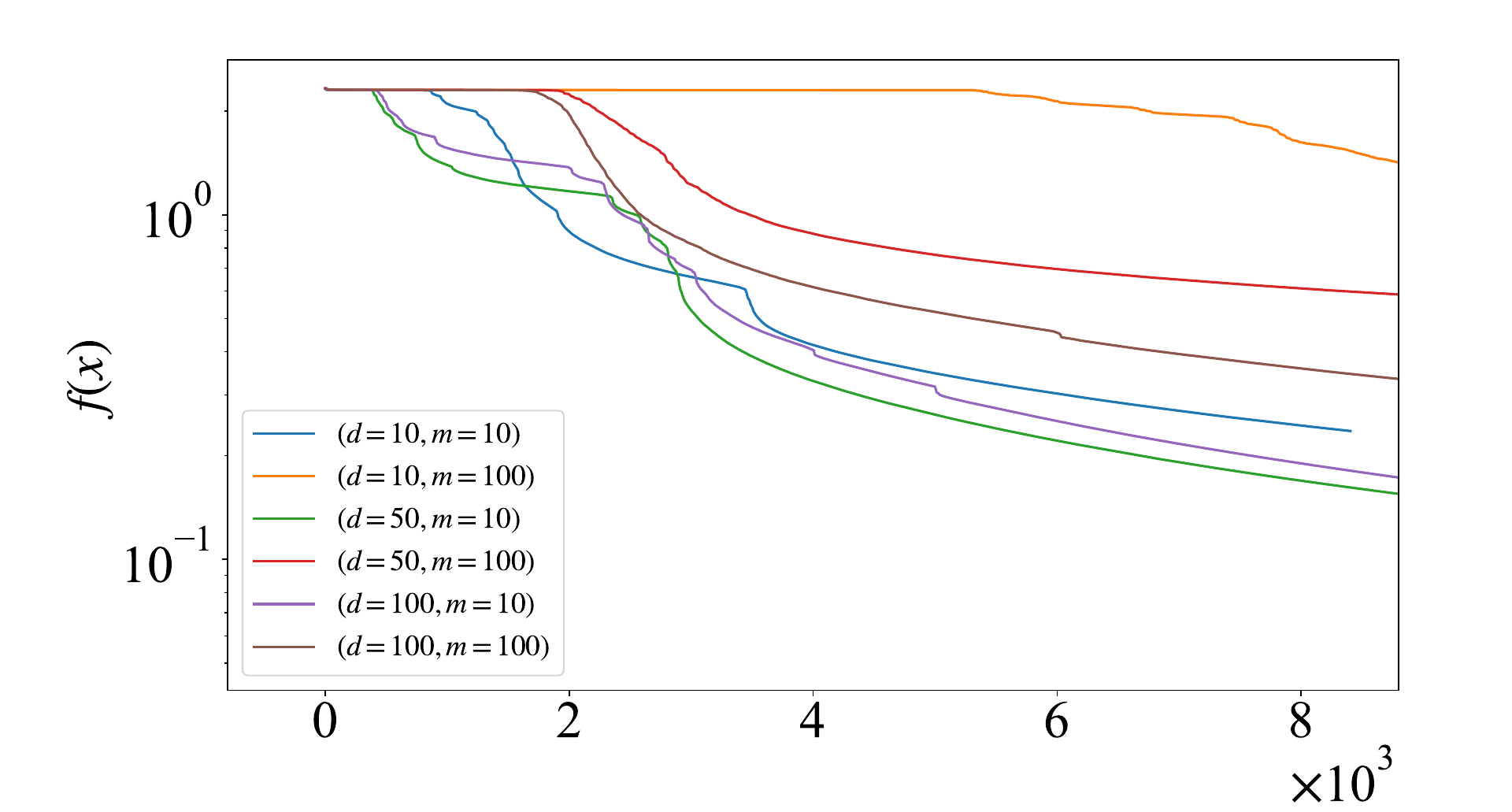}
		\subcaption{Time vs. function value}
		\label{fig:compare proposed method parameters(time)}
	\end{minipage}
	\caption{Results of the proposed method (average for 10 runs) for LNN.
		For the readability, this figure shows only the average, but the plots showing (the average $\pm$ the standard deviation for 10 runs) are shown in Figure~\ref{fig:parameter_m_d} in supplementary material. The performance of all 10 runs for each $m$ and $d$ is also shown in Figure~\ref{fig:parameter_m_d_allrun}.}
	\label{fig:compare proposed method parameters}
\end{figure}
The experiments confirm that Algorithm \ref{alg7} converges faster for larger value of $d$ in terms of iteration, which is indeed in agreement with \eqref{eq:nonconvex2} and \eqref{eq:nonconvex2b}.
However, it can be noticed that the algorithm performs better for smaller values of $m$.
This is because our method estimates three gradients per iteration. Hence we need to compute $2m + d$  directional derivatives: $\nabla \tilde{g}_k(\bm{0})$, $\nabla g_k(\bm{0})$ and $\nabla g_k(\alpha_k\bm{d}_k)$. 
Although the convergence rates in Theorems~\ref{nonconvex2}-\ref{PL2} (e.g., as implied from \eqref{decrease2}) depend on the ratio $\frac{d}{n}$,
Figure~\ref{fig:compare proposed method parameters(iteration)} indicates almost no difference in the performance. A small value of $m$ has a more significant impact on the speed-up of the algorithm than the $\frac{d}{n}$.

\subsection{Comparison of the proposed method with previous studies}\label{comparepeform}
In this section, we compare our methods with existing algorithms shown below: \\
\textbf{Proposed}: proposed method (Algorithm \ref{alg7} or the modified one: Algorithm \ref{alg8}) \\
\textbf{GD}: The steepest descent method using Armijo linear search \\ 
\textbf{AGD}: Accelerated gradient descent method using Armijo linear search 
and restart. \\
\textbf{LMN}: Algorithm of \cite{Comdir} using \eqref{P_kBFGS} for $P_k$ \\
\textbf{SGD}: Subspace gradient descent of \cite{kozak2021stochastic} using Gaussian random matrices. \\
\textbf{RSRNM}: Randomized subspace regularized Newton method in \cite{R-subspace}

The existing random subspace BFGS, \cite{R-BFGS}, is not considered for comparison because the matrix-to-matrix products cannot be performed,  due to memory limitation, for the dimension $n$ considered.
Indeed, the method needs to store an $n\times n$ matrix, which is impracticable when $n$ is too large.

The comparison results are shown in Figure \ref{fig:realworld_data} (for more details, see Figures~\ref{fig:errorbar_mlpnet} and \ref{fig:errorbar_cnn} in supplementary material).
In Figure~\ref{fig:result_convolutional_neural_network}, the existing second-order method did not work because Hessian-vector products cause memory issues.
Hence, we compared the proposed method and the existing first-order methods.

\begin{figure}[tb]
	\begin{minipage}[t]{0.5\hsize}
		\centering
		\includegraphics[width=3.0 in]{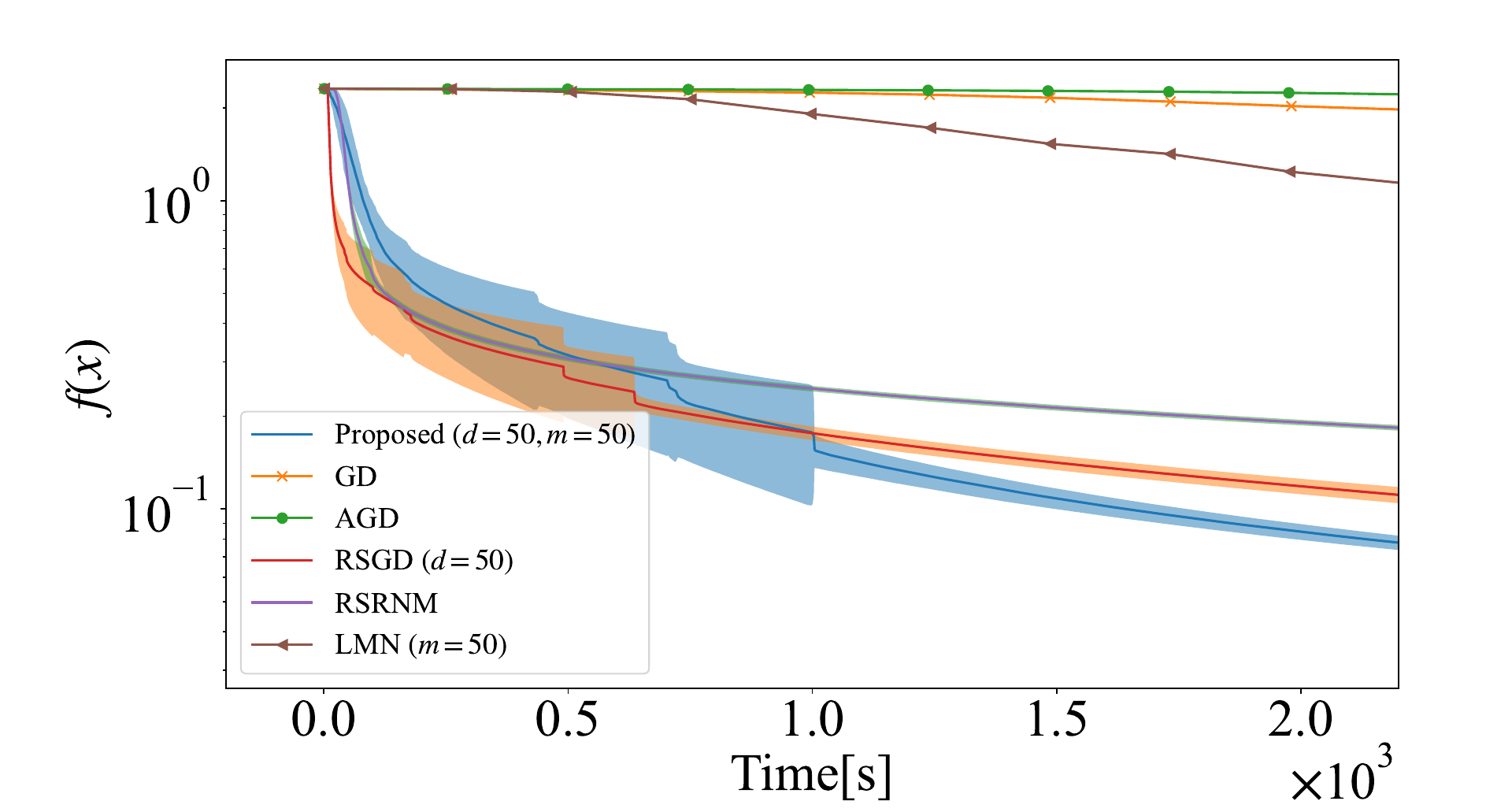}
		\subcaption{LNN}
		\label{fig:result_linear_neural_network}
	\end{minipage} 
	\begin{minipage}[t]{0.5\hsize}
		\centering
		\includegraphics[width=3.0 in]{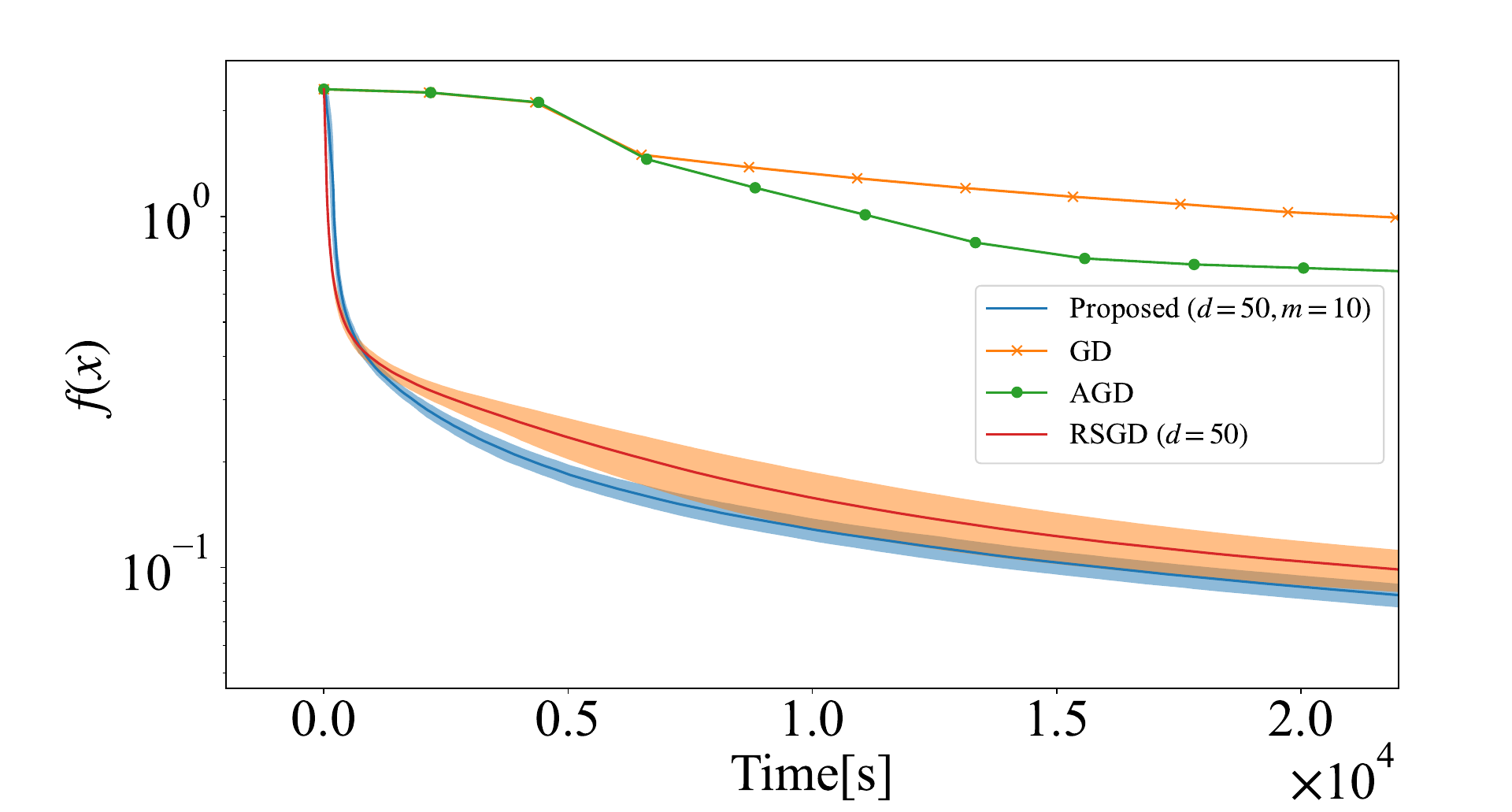}
		\subcaption{CNN with finite difference gradients}
		\label{fig:result_convolutional_neural_network}
	\end{minipage} 
	\caption{The average $\pm$ the standard deviation for LNN (10 runs) and CNN (5 runs).}
	\label{fig:realworld_data}
\end{figure}

The proposed method performs well over time for all functions, especially in terms of function values. 
The worst-case iteration complexity is comparable to that of other methods, confirming that the advantage of reducing the computational complexity per iteration 
is fully exploited. Fast function-value reduction is achieved for both convex and nonconvex functions, and the method is useful for a wide range of functions. 
\section{Conclusion} \label{conclusion}
 We proposed a novel quasi-Newton algorithm that significantly reduces the computational complexity per iteration due to the computation on the subspace, 
while maintaining the same order of the worst-case iteration complexities in expectation, compared to the previous works.

\paragraph{Limitation and future work}
Nevertheless, much of the computation time is spent on 
the computation of the following subspace gradients: $\nabla \tilde{g}_k(\bm{0})=Q_k^\top \nabla f(\bm{x}_k)$, $\nabla g_k(\bm{0})=P_k^\top \nabla f(\bm{x}_k)$ and $\nabla g_k(\alpha_k\bm{d}_k)$. We would like to reduce the amount of calculation for subspace gradients.
As discussed in Section~\ref{prepare}, a quite recent paper \cite{jin2024globalconvBFGS} proved the iteration complexity without modifying the matrix of the BFGS formula by adopting the exact line search rule. We would like to consider the subspace variant of the method while the exact line search is not so practical.
Another limitation in our paper would be that the function is assumed to be smooth ($f \in C^2$). It would be interesting, in some future work, to extend this approach to a non-smooth setting.


\bibliographystyle{plain}
\bibliography{ref}

\newpage
\appendix


\section{Proposed Algorithm using Finite-difference Computation for Gradient Vectors}
In this section, we summarized the modifications done to the matrix $P_k$ and to Algorithm \ref{alg7} using the notation of Section \ref{subsect:2}.
\begin{equation}\label{approxP_kbis}
	P_k=\left[\dfrac{\bm{x}_{k-\frac{m}{2}+1}}{\|\bm{x}_{k-\frac{m}{2}+1}\|},\dfrac{Q_{k-\frac{m}{2}+1}{\nabla}_{Q_{k-\frac{m}{2}+1},\varepsilon}f(\bm{x}_{k-\frac{m}{2}+1})}{\|Q_{k-\frac{m}{2}+1}{\nabla}_{Q_{k-\frac{m}{2}+1},\varepsilon}f(\bm{x}_{k-\frac{m}{2}+1})\|},\dots,\dfrac{\bm{x}_k}{\|\bm{x}_k\|},\dfrac{Q_k{\nabla}_{Q_k,\varepsilon}f(\bm{x}_k)}{\|Q_k{\nabla}_{Q_k,\varepsilon}f(\bm{x}_k)\|}\right]
\end{equation}
The modified variant of Algorithm \ref{alg7} is described as follows.
\begin{algorithm}[hbt!]
	\caption{Subspace Quasi-Newton Method with Randomly Approximated Gradient computed with finite difference}
	\label{alg8}
	Input: $\bm{x}_0$,  $P_0,\ M_1I_m\preceq H_0\preceq M_2I_m$, $\beta,c\in(0,\frac{1}{2})$
	
	parameters: $M_1,M_2>0$, $\varepsilon$ .
	\begin{algorithmic}[1]
		\For {$k=0,1,\dots$}
		\State Sample  $(Q_k)_{ij}\ \sim\ N(0,1)\qquad(i=1,2,\dots,n\quad j=1,2,\dots,d)$.
		
		\State Compute  ${\nabla}_{Q_k,\varepsilon}f(\bm{x}_k)$ using \eqref{finiteDiff_Q}.
		\State Set $P_{k}$ by using \eqref{approxP_kbis} with $\bm{x}_{k}$ and $Q_k {\nabla}_{Q_k,\varepsilon}f(\bm{x}_k)$, and construct $g_k$ as in \eqref{subspaceprob}.
		\State $\bm{d}_k\ \leftarrow\ -H_k{\nabla}_{P_k,\varepsilon}f(\bm{x}_k)$ 
		\For{$i=0,1,\dots$}  
		\State $\alpha_k\ \leftarrow\ \beta^i$
		\If {$f(\bm{x}_k+\alpha_kP_k\bm{d}_k)\leq g_k(\bm{0})+c\alpha_k{\nabla}_{P_k,\varepsilon}f(\bm{x}_k)^{\top}\bm{d}_k$} 
		\State break
		\EndIf
		\EndFor               
		\State $\bm{x}_{k+1}\ \leftarrow\   \bm{x}_k+\alpha_kP_k\bm{d}_k$
		\State $\tilde{H}_{k+1}\ \leftarrow\ \textup{BFGS}(H_k,\alpha_k\bm{d}_k,{\nabla}_{P_k,\varepsilon}f(\bm{x}_k+\alpha_k \bm{d}_k)-{\nabla}_{P_k,\varepsilon}f(\bm{x}_k))$
		\State $H_{k+1}\ \leftarrow\ $Modify\_eig$(\tilde{H}_{k+1})$
		\EndFor 
	\end{algorithmic}
\end{algorithm}

\section{Proofs}\label{sec:allproof}
 Let us define the event $\mathcal{A}(\epsilon)$ such that both Lemmas \ref{rand_norm2} and \ref{rand_norm} hold.
\subsection{Proof of Theorem \ref{stepsize2}} \label{sec:stepsize_proof}
	By definition of $H_k$, we have that $\bm{d}_k=-H_k P_k^{\top}\nabla f(\bm{x}_k)$ and $H_k \preceq M_2I_m$  (by \eqref{B_knorm}). Therefore, 
	\begin{equation}\label{d_kgrad}
		-(P_k\bm{d}_k)^{\top}\nabla f(\bm{x}_k)=-\bm{d}_k^{\top}P_k^{\top}\nabla f(\bm{x}_k)=\bm{d}_k^{\top}H_k^{-1}\bm{d}_k\geq\dfrac{1}{M_2}\|\bm{d}_k\|^2.
	\end{equation}
	From Lemma \ref{P_knorm}, we have
	\begin{equation}\label{P_kd_k}
		\|\bm{d}_k\|^2\geq\dfrac{1}{m}\|P_k\bm{d}_k\|^2.
	\end{equation}
	Hence, we obtain
	\begin{equation}
		-(P_k\bm{d}_k)^{\top}\nabla f(\bm{x}_k)\overset{\eqref{d_kgrad}}{\geq}\dfrac{1}{M_2}\|\bm{d}_k\|^2\overset{\eqref{P_kd_k}}{\geq}\dfrac{1}{mM_2}\|P_k\bm{d}_k\|^2,
		\notag
	\end{equation}
	which leads to 
	\begin{equation}\label{d_kgrad2}
		\|P_k\bm{d}_k\|^2\leq -mM_2(P_k\bm{d}_k)^{\top}\nabla f(\bm{x}_k).
	\end{equation}
	Here, from \cite[Lemma 1.2.3]{Lecture}, the following holds from Assumption \ref{Lipdif2}.
	\begin{equation}
		f(\bm{x}+\bm{y})\leq f(\bm{x})+\nabla f(\bm{x})^{\top}\bm{y}+\dfrac{L}{2}\|\bm{y}\|^2\ (\forall \bm{x},\bm{y}\in\mathbb{R}^n).
		\notag
	\end{equation}
	Using this and \eqref{d_kgrad2}, we have that
	\begin{equation}
		\begin{split}
			f(\bm{x}_k+\alpha_k P_k\bm{d}_k)&\leq f(\bm{x}_k)+\alpha_k\nabla f(\bm{x}_k)^{\top}P_k\bm{d}_k+\alpha_k^2\dfrac{L}{2}\|P_k\bm{ d}_k\|^2\\
			&\overset{\eqref{d_kgrad2}}{\leq} f(\bm{x}_k)+\alpha_k\nabla f(\bm{x}_k)^{\top}P_k\bm{d}_k\left(1-\dfrac{mM_2L\alpha_k}{2} \right).
		\end{split}
		\notag
	\end{equation}
	Therefore, since $
	\nabla f(\bm{x}_k)^{\top}P_k\bm{d}_k=\nabla g_k(\bm{0})^{\top}\bm{d}_k <0
	$, we see that if
	\begin{equation}
		\left(1-\dfrac{mM_2L\alpha_k}{2}\right)\geq c, 
		\notag
	\end{equation}
	i.e., $\alpha_k\leq\dfrac{2(1-c)}{mM_2L}$ holds, 
	{\eqref{linsearch}} is satisfied. We conclude the proof by noticing $0<\beta<1$.

\subsection{Proof of Lemma \ref{P_knorm}}
\begin{proof}
	Since the operator norm of the matrix is equal to the maximum singular value of the matrix and transposing the matrix does not change the singular values except for \textup{0}, we have that $\|P\|=\|P_k^{\top}\|$. If $\|\bm{v}\|=1$, then we have 
	\begin{equation}
		|(P_k^{\top}\bm{v})_i|\leq1\Rightarrow\|P_k^{\top}\bm{v}\|\leq \sqrt{m},
		\notag
	\end{equation}
	since  each column of $P_k$ is normalized. 
\end{proof}

\subsection{Proofs for Exact Subspace Gradient Computation}

\textbf{Proof of Theorem \ref{nonconvex2}}
\begin{proof}
	From the expression \eqref{B_knorm} and \eqref{linsearch}, we obtain
	\begin{equation}\label{decrease}
		\begin{split}
			f(\bm{x}_{k+1})-f(\bm{x}_k)&\overset{\eqref{linsearch}}{\leq} -c\alpha_k\nabla f(\bm{x}_k)^{\top}P_k H_k P_k^{\top}\nabla f(\bm{x}_k)\\
			&\overset{\eqref{B_knorm}}{\leq} -c\alpha_kM_1\|P_k^{\top}\nabla f(\bm{x}_k)\|^2.
		\end{split}
	\end{equation}
	Since the columns of $P_k$ contain
	$\dfrac{Q_kQ_k^{\top}\nabla f(\bm{x}_k)}{\|Q_kQ_k^{\top}\nabla f(\bm{x}_k)\|}$,
	where $Q_k$ is a random matrix, we have that 
	\begin{equation}\label{P_knabla}
		\begin{split}
			\|P_k^{\top}\nabla f(\bm{x}_k)\|^2&\geq \left(\nabla f(\bm{x}_k)^{\top}\dfrac{Q_kQ_k^{\top}\nabla f(\bm{x}_k)}{\|Q_kQ_k^{\top}\nabla f(\bm{x}_k)\|}\right)^2\\
			&\geq\left(\dfrac{\|Q_k^{\top}\nabla f(\bm{x}_k)\|^2}{\|Q_kQ_k^{\top}\|\|\nabla f(\bm{x}_k)\|}\right)^2\\
			&\overset{\eqref{vectornorm},\eqref{oparationnorm}}{\geq}\left(\dfrac{(1-\epsilon)\|\nabla f(\bm{x}_k)\|^2}{\mathcal{C}\dfrac{n}{d}\|\nabla f(\bm{x}_k)\|}\right)^2\\
			&=\left(\dfrac{(1-\epsilon)d}{\mathcal{C}n}\right)^2\|\nabla f(\bm{x}_k)\|^2 
		\end{split}
	\end{equation}
	is satisfied with probability at least $p(\epsilon)$ (c.f. \eqref{peps}). Let $\mathcal{A}$ be the event corresponding to \eqref{P_knabla}. By the above, we have that $\mathbb{P}(\mathcal{A})\ge p(\epsilon)$.
	Hence, when $\omega \in \mathcal{A}$ (here $\omega$ denotes the randomness associated with $Q_k$), from \eqref{decrease} we have
	\begin{equation}\label{decrease3}
		f(\bm{x}_{k+1})-f(\bm{x}_k)\leq-c\alpha_kM_1\left(\dfrac{(1-\epsilon)d}{\mathcal{C}n}\right)^2\|\nabla f(\bm{x}_k)\|^2.
	\end{equation}
	Notice that when $\omega \notin \mathcal{A}$ $f(\bm{x}_{k+1})-f(\bm{x}_k)\leq 0$ holds by linear search, and therefore, given $\bm{x}_k$, taking the expectation for $f(\bm{x}_{k+1})-f(\bm{x}_k)$, conditionally to $\mathcal{A}$, leads to
	\begin{equation}
		\begin{split}
			\mathbb{E}[f(\bm{x}_{k+1})-f(\bm{x}_k)|\bm{x}_k]&\leq -c\alpha_kM_1\left(\dfrac{(1-\epsilon)d}{\mathcal{C}n}\right)^2\|\nabla f(\bm{x}_k)\|^2\times p(\epsilon)+0\times(1-p(\epsilon))\\
			&\overset{\mbox{\small Theorem }\ref{stepsize2}}{\leq} -c\bar{\alpha} M_1\left(\dfrac{(1-\epsilon)d}{\mathcal{C}n}\right)^2 p(\epsilon) \|\nabla f(\bm{x}_k)\|^2.
		\end{split}  
		\notag
	\end{equation}
	Taking expectation regarding 
	$\bm{x}_k$ for the above, 
	\begin{equation}\label{decrease2}
		\mathbb{E}[f(\bm{x}_{k+1})-f(\bm{x}_k)]\leq-c\bar{\alpha} M_1\left(\dfrac{(1-\epsilon)d}{\mathcal{C}n}\right)^2p(\epsilon)\mathbb{E }[\|\nabla f(\bm{x}_k)\|^2].
	\end{equation}
	Adding these equations for $k=0,1,\dots$, we obtain 
	\begin{equation}
		\sum_{j=0}^{k}c\bar{\alpha} M_1\left(\dfrac{(1-\epsilon)d}{\mathcal{C}n}\right)^2p(\epsilon)\mathbb{E}[\|\nabla f(\bm{x}_j)\|^2]\leq \mathbb{E}[f(\bm{x}_0)-f(\bm{x}_k)]\leq f(\bm{x}_0)-f^*,
		\notag
	\end{equation}
	which is rewritten as
	\begin{equation}
		\sum_{j=0}^{k}\mathbb{E}[\|\nabla f(\bm{x}_j)\|^2]\leq \left(c\bar{\alpha} M_1\left(\dfrac{(1-\epsilon)d}{\mathcal{C}n }\right)^2p(\epsilon)\right)^{-1}(f(\bm{x}_0)-f^*).
		\notag
	\end{equation}
	From
	\begin{equation}\label{gradsquare}
		\begin{split}
			\mathbb{E}\left[\min_{0\leq j\leq k}\|\nabla f(\bm{x}_k)\|^2\right]&\leq\dfrac{1}{k+1}\sum_{j=0}^{k}\mathbb{E}[\|\nabla f(\bm{x}_k)\|^2], \\
		\end{split}
	\end{equation}
	and by Jensen's inequality, we deduce that 
	\[  \mathbb{E}\left[\sqrt{\min_{0\leq j\leq k}\|\nabla f(\bm{x}_j)\|^2}\right] \le \sqrt{\mathbb{E}\left[\min_{0\leq j\leq k}\|\nabla f(\bm{x}_j)\|^2\right]} \le \sqrt{\frac{1}{k+1} \left(c\bar{\alpha} M_1\left(\dfrac{(1-\epsilon)d}{\mathcal{C}n }\right)^2p(\epsilon)\right)^{-1}(f(\bm{x}_0)-f^*)} \]
	which concludes the proof noticing that 
	\[ \mathbb{E}\left[\sqrt{\min_{0\leq j\leq k}\|\nabla f(\bm{x}_j)\|^2}\right] = \mathbb{E}\left[\min_{0\leq j\leq k}\|\nabla f(\bm{x}_j)\|\right].\] Also from \cite[Proposition 3.4]{convergence}, we have that
	$\mathbb{E}[\|\nabla f(\bm{x}_k)\|]$ converges to $0$.
\end{proof}

\textbf{Proof of Theorem \ref{convex2}}
\begin{proof}
	From the convexity of $f$ and by definition of $R_0$, we have, by \eqref{R_0}, that for any $\bm{x}\in\mathbb{R}^n$ and any optimal solution $\bm{x}^*$: 
	\begin{equation}
		f(\bm{x})-f^*\leq\nabla f(\bm{x})^\top(\bm{x}-\bm{x}^*)\leq\|\nabla f(\bm{x})\| \|\bm{x}-\bm{x}^*\|\leq\|\nabla f(\bm{x})\|R_0,
		\notag
	\end{equation}
	that is, 
	\begin{equation}\label{f(x)}
		\|\nabla f(\bm{x})\|^2\geq\dfrac{(f(\bm{x})-f^*)^2}{R_0^2}.
	\end{equation}
	Let us now define
	\begin{equation}
		\Delta_k:=\mathbb{E}[f(\bm{x}_k)-f^*].
		\notag
	\end{equation}
	From Jensen's inequality, it follows that
	\begin{equation}
		\mathbb{E}[(f(\bm{x}_k)-f^*)^2]\geq\Delta_k^2.
		\notag
	\end{equation}
	Setting $\bm{x}=\bm{x}_k$ and taking the expectation of both sides of \eqref{f(x)}, we obtain
	\begin{equation}\label{minusnabla}
		\mathbb{E}[\|\nabla f(\bm{x}_k)\|^2]\geq\dfrac{\mathbb{E}[(f(\bm{x}_k)-f^*)^2]}{R_0^2}\geq\dfrac{\Delta_k^2}{R_0^2}
	\end{equation}
	Rewriting \eqref{decrease2} using $\Delta_k$, 
	we obtain
	\begin{eqnarray*}
		\Delta_{k+1} & \leq & \Delta_k-c\bar{\alpha} M_1\left(\dfrac{(1-\epsilon)d}{\mathcal{C}n}\right)^2p(\epsilon)\mathbb{E}[\|\nabla f(\bm{x}_k )\|^2] \\
		& \overset{\eqref{minusnabla}}{\leq} & \Delta_k-c\bar{\alpha} M_1p(\epsilon)\left(\dfrac{(1-\epsilon)d}{\mathcal{C}nR_0}\right)^2\Delta_k^2
		\notag
	\end{eqnarray*}
	Dividing this equation by both sides $\Delta_k\Delta_{k+1}(>0)$ and using the fact that $\Delta_k$ is monotonically decreasing with respect to $k$ and therefore $\dfrac{\Delta_k}{\Delta_{k+1}}\geq 1$, 
	\begin{equation}
		\dfrac{1}{\Delta_k}\leq\dfrac{1}{\Delta_{k+1}}-c\bar{\alpha} M_1p(\epsilon)\left(\dfrac{(1-\epsilon)d}{\mathcal{C}nR_0}\right)^2\dfrac{\Delta_k }{\Delta_{k+1}}\leq\dfrac{1}{\Delta_{k+1}}-c\bar{\alpha} M_1p(\epsilon)\left(\dfrac{(1-\epsilon)d}{\mathcal{C}nR_0}\right)^2.
		\notag
	\end{equation}
	By using this with $k=0,1,\dots$, we get
	\begin{equation}
		\dfrac{1}{\Delta_0}\leq\dfrac{1}{\Delta_{k+1}}-cM_1p(\epsilon)\left(\dfrac{(1-\epsilon)d}{\mathcal{C}nR_0}\right)^2 (k+1)\bar{\alpha} 
		\notag
	\end{equation}
	By writing this, 
	\begin{equation}\label{Delta}
		\begin{split}
			\Delta_{k+1}\leq 
			\dfrac{\Delta_0\mathcal{C}^2n^2R_0^2}{\mathcal{C}^2n^2R_0^2+\Delta_0cM_1p(\epsilon)(1-\epsilon)^2d^2(k+1)\bar{\alpha }}.
		\end{split}
	\end{equation}
	Now, from the definition of $\Delta_k$, \eqref{Lipdif} and \eqref{R_0} together with the convexity of $f$ and
	$\nabla f(\bm{x}^*)=\bm{0}$, we have
	\begin{equation}\label{Delta_0}
		\Delta_0=f(\bm{x}_0)-f^*\leq L \|\bm{x}_0-\tilde{\bm{x}}\|^2 \leq L R_0^2,
	\end{equation}
	where
	\begin{equation}
		\tilde{\bm{x}}:=\argmin_{\bm{x}^* \in \Omega}\|\bm{x}_0-\bm{x}^*\|.
		\notag
	\end{equation}
	Since the term in the middle of inequalities \eqref{Delta} is monotonically increasing with respect to $\Delta_0$,
	\eqref{Delta} together with  \eqref{Delta_0} leads to
	\begin{equation}
		\Delta_{k}\leq 
		\dfrac{L \mathcal{C}^2n^2R_0^2}{\mathcal{C}^2n^2+L cM_1p(\epsilon)(1-\epsilon)^2d^2k\bar{\alpha}},
		\notag
	\end{equation}
	which completes the proof. 
\end{proof}

\textbf{Proof of Theorem \ref{PL2}}
\begin{proof}
	Using \eqref{decrease3} and \eqref{PL}, we obtain
	\begin{equation}
		\begin{split}
			f(\bm{x}_{k+1})-f^*&\leq(f(\bm{x}_k)-f^*)-c\alpha_kM_1\left(\dfrac{(1-\epsilon)d}{\mathcal{C}n}\right)^2\|\nabla f(\bm{x}_k)\|^2\\
			&\overset{\eqref{PL}}{\leq}\left(1-2\sigma c\alpha_kM_1\left(\dfrac{(1-\epsilon)d}{\mathcal{C}n}\right)^2\right)(f(\bm{x}_k)-f^*)\\
			&\overset{\mbox{Theorem } \ref{stepsize2}}{\leq}\left(1-2\sigma c\bar{\alpha}M_1\left(\dfrac{(1-\epsilon)d}{\mathcal{C}n}\right)^2\right)(f(\bm{x}_k)-f^*)    
		\end{split}
		\notag
	\end{equation}
	with probability at least $p(\epsilon)$. Let us denote by $\mathcal{A}$ the event on which the above inequality holds. Since in any case $f(\bm{x}_{k+1})-f^*\leq f(\bm{x}_{k})-f^*$, we have that, by taking the conditional expectation with respect to $x_k$: 
	\begin{equation}
		\begin{split}
			\mathbb{E}[f(\bm{x}_{k+1})-f^*|\bm{x}_k]&\leq \left(p(\epsilon)\left(1-2\sigma c\bar{\alpha}M_1\left(\dfrac{(1-\epsilon)d}{\mathcal{C}n}\right)^2\right)+(1-p(\epsilon))\right)(f(\bm{x}_k)-f^*)\\
			&=\left(1-2\sigma c\bar{\alpha}M_1p(\epsilon)\left(\dfrac{(1-\epsilon)d}{\mathcal{C}n}\right)^2\right)(f(\bm{x}_k)-f^*).    
		\end{split}
		\notag
	\end{equation}
	By taking the expectation, with respect to $\bm{x}_k$, in both sides of the above inequality we obtain
	\begin{equation}
		\mathbb{E}[f(\bm{x}_{k+1})-f^*]=\left(1-2\sigma c\bar{\alpha}M_1p(\epsilon)\left(\dfrac{(1-\epsilon)d}{\mathcal{C}n}\right)^2\right)\mathbb{E }[f(\bm{x}_k)-f^*].
		\notag
	\end{equation}
	Since $2\sigma c\bar{\alpha}M_1\left(\dfrac{(1-\epsilon)d}{\mathcal{C}n}\right)^2>0$,
	the coefficient of $\mathbb{E}[f(\bm{x}_k)-f^*]$ is less than 1. Therefore, $\mathbb{E}[f(\bm{x}_k)-f^*]$ converges  linearly to $0$.
\end{proof}

\subsection{Proofs for Subspace Gradient Approximation by Finite Difference}
We have the following approximation result:
\begin{proposition}\label{prop:approxnorm1} Under Assumption~\ref{Lipdif2}
	$$\|\nabla_{Q,\varepsilon}f(\bm{x})-Q^\top \nabla f(\bm{x})\| \le 2d \|Q\|^2L \varepsilon .$$
\end{proposition}
\begin{proof}
	Let $1\le i\le d $. By a Taylor expansion of $f$ around $\bm{x}_k-\epsilon Q \bm{e}_i$, we have that there exists $\tilde{\bm{x}}\in [\bm{x}-\varepsilon Q\bm{e}_i,\bm{x}+\varepsilon Q\bm{e}_i]$ such that 
	$$f(\bm{x}+\varepsilon Q\bm{e}_i)-f(\bm{x}-\varepsilon Q\bm{e}_i)=2\varepsilon (Q\bm{e}_i)^\top \nabla f(\tilde{\bm{x}}). $$
	Hence, by Assumption~\ref{Lipdif2}, we have that
	$$\left|{\nabla}_{Q,\varepsilon}f(\bm{x})_i - (Q^\top \nabla f(\bm{x}))^\top \bm{e}_i \right| \le 2\varepsilon L \|Q\|^2. $$
	We deduce therefore that
	$$ \left\|{\nabla}_{Q,\varepsilon}f(\bm{x}) - Q^\top \nabla f(\bm{x}) \right\| \le  2d \|Q\|^2L \varepsilon .$$
	
\end{proof}
\begin{proposition}\label{prop:approxnorm2} Under Assumption~\ref{Lipdif2}
	$$\|\nabla_{P,\varepsilon}f(\bm{x})-P^\top \nabla f(\bm{x})\| \le 2m^2L \varepsilon .$$
\end{proposition}
\begin{proof}
	The proof is the same as the one of the previous proposition; we only use the fact that 
	$\|P\|\le \sqrt{m}$, by Lemma~\ref{P_knorm}.
\end{proof}

\begin{lem}\label{lem:aux}
	Assume that $\|\nabla f(\bm{x}_k)\|\ge\frac{ 4 \sqrt{d}\|Q_k\|^2L \varepsilon}{1-\epsilon}$. Then	we have that
	$$	\|P_k^\top \nabla f(\bm{x}_k)\| \ge \dfrac{d(1-\epsilon)}{3\mathcal{C}n} \|\nabla f(x_k)\| $$
	holds with probability at least $p(\epsilon)$.
\end{lem}
\begin{proof}	Assume first that $\|Q_k^\top \nabla f(\bm{x}_k)\|\ge 4d \|Q_k\|^2L \varepsilon$ .
	Since the columns of now contains $P_k$ contain $\dfrac{Q_k\nabla_{Q_k,\varepsilon}f(\bm{x}_k)}{\|Q_k\nabla_{Q_k,\varepsilon}f(\bm{x}_k)\|}$, we have
	$$ 	\|P_k^\top \nabla f(\bm{x}_k)\| \ge  \frac{(Q_k^\top \nabla f(\bm{x}_k))^\top \nabla_{Q_k,\varepsilon}f(\bm{x}_k)}{ \|Q_k\nabla_{Q_k,\varepsilon}f(\bm{x}_k)\|}.$$
	Hence, by Proposition \ref{prop:approxnorm1} and using that :
	$$ 	\|P_k^\top \nabla f(\bm{x}_k)\| \ge  \frac{\|Q_k^\top \nabla f(\bm{x}_k)\|(\|Q_k^\top \nabla f(\bm{x}_k)\| - 2d \|Q_k\|^2L \varepsilon )}{\|Q_k \|( \|Q_k^\top f(\bm{x}_k)\| + 2d \|Q_k\|^2L \varepsilon )} >0.$$
	Using that $\|Q_k^\top \nabla f(\bm{x}_k)\|\ge 4d \|Q_k\|^2L \varepsilon$, we deduce that
	$\|Q_k^\top \nabla f(\bm{x}_k)\| - 2d \|Q_k\|^2L \varepsilon \ge \frac{1}{2} \|Q_k^\top \nabla f(\bm{x}_k)\| $ and that $\|Q_k^\top f(\bm{x}_k)\| + 2d \|Q_k\|^2L \varepsilon \le \frac{3}{2}\|Q_k^\top f(\bm{x}_k)\| $. Hence,
	$$ 	\|P_k^\top \nabla f(\bm{x}_k)\| \ge  \frac{1}{3}\frac{\|Q_k^\top \nabla f(\bm{x}_k)\|}{ \|Q_k\|}.$$
	Similarly to \eqref{P_knabla}, we deduce that with probability at least $p(\epsilon)$:
	$$ 	\|P_k^\top \nabla f(\bm{x}_k)\| \ge \dfrac{d(1-\epsilon)}{3\mathcal{C}n} \|\nabla f(x_k)\|.$$ 
	Notice that we have by Lemma \ref{P_knorm} that $\|Q^\top \nabla f(\bm{x}_k)\| \ge (1-\epsilon)\sqrt{d}\nabla f(\bm{x}_k)\| $. Hence $\|\nabla f(\bm{x}_k)\|\ge\frac{ 4 \sqrt{d}\|Q_k\|^2L \varepsilon}{1-\epsilon}$ implies that $\|Q_k^\top \nabla f(\bm{x}_k)\|\ge 4d \|Q_k\|^2L \varepsilon$.
\end{proof}

Notice first we have the same proposition for the lower bound on the step-size:
\begin{lem}\label{stepsize4}
	Assume that, $c<1/2$ and
	$$\|{\nabla}f(\bm{x}_k)  \| \ge  \dfrac{6\mathcal{C}nm^2L}{d} \left(\frac{M_2}{M_1 c}+1 \right)\varepsilon.$$ Suppose that Assumption~\ref{Lipdif2} is satisfied. For any $k$, the step size $\alpha_k$ satisfies, with probability at least $p(\epsilon)$ 
	\begin{equation}
		\alpha_k\geq\bar{\alpha}:=\min\left(1,\dfrac{2\beta(1-2c)}{mM_2L}\right).
	\end{equation}
\end{lem}

\begin{proof}
	By definition of $H_k$, we have that $\bm{d}_k=-H_k{\nabla}_{P_k,\varepsilon}f(\bm{x}_k)$ and $H_k \preceq M_2I_m$  (by \eqref{B_knorm}). Therefore, 
	\begin{equation}
		-\bm{d}_k^{\top}{\nabla}_{P_k,\varepsilon}f(\bm{x}_k)=\bm{d}_k^{\top}H_k^{-1}\bm{d}_k\geq\dfrac{1}{M_2}\|\bm{d}_k\|^2.
	\end{equation}
	which leads to 
	\begin{equation}
		\|\bm{d}_k\|^2\leq -M_2\bm{d}_k^{\top}{\nabla}_{P_k,\varepsilon}f(\bm{x}_k).
	\end{equation}
	Hence, we have that
	\begin{equation}
		\begin{split}
			f(\bm{x}_k+\alpha_k P_k\bm{d}_k)&\leq f(\bm{x}_k)+\alpha_k\nabla f(\bm{x}_k)^{\top}P_k\bm{d}_k+\alpha_k^2\dfrac{L}{2}\|P_k\bm{ d}_k\|^2\\
			&\leq f(\bm{x}_k)+\alpha_k\nabla_{P_k,\varepsilon} f(\bm{x}_k)^{\top}\bm{d}_k+\alpha_k^2m\dfrac{L}{2}\|\bm{ d}_k\|^2+\alpha_k\|\nabla_{P_k,\varepsilon}f(x)-P_k^\top \nabla f(x)\|\|\bm{d}_k\|\\
			&\leq f(\bm{x}_k)+\alpha_k\nabla_{P_k,\varepsilon} f(\bm{x}_k)^{\top}\bm{d}_k+\alpha_k^2\dfrac{L}{2}m\|\bm{ d}_k\|^2+2m^2L\alpha_k\varepsilon\|\bm{d}_k\|\\
			&{\leq} f(\bm{x}_k)+\alpha_k\nabla f(\bm{x}_k)^{\top}P_k\bm{d}_k\left(1-\dfrac{mM_2L\alpha_k}{2} \right)+\alpha_k^2\dfrac{L}{2}m\|\bm{ d}_k\|^2+2m^2L\alpha_k\varepsilon\|\bm{d}_k\|.
		\end{split}
		\notag
	\end{equation}
	Assume first that $\|{\nabla}_{P_k,\varepsilon}f(\bm{x}_k)\|\ge 2m^2L\frac{M_2}{M_1 c}\varepsilon $ holds. Then it implies that $2m^2L\varepsilon \le c\frac{1}{M_2}\|\bm{d}_k\|$. Hence,
	$$f(\bm{x}_k+\alpha_k P_k\bm{d}_k) \le f(\bm{x}_k)+\alpha_k\nabla f(\bm{x}_k)^{\top}P_k\bm{d}_k\left(1-\dfrac{mM_2L\alpha_k}{2} -{c}\right).$$
	Therefore, since $
	\nabla f(\bm{x}_k)^{\top}P_k\bm{d}_k=\nabla g_k(\bm{0})^{\top}\bm{d}_k <0
	$, we see that if
	\begin{equation}
		\left(1-{c}-\dfrac{mM_2L\alpha_k}{2}\right)\geq c, 
		\notag
	\end{equation}
	i.e., $\alpha_k\leq\dfrac{2(1-2c)}{mM_2L}$ holds, 
	{\eqref{linsearch}} is satisfied. 
	To conclude, we prove that the condition in the statement of the Lemma implies that 
	$$\|{\nabla}_{P_k,\varepsilon}f(\bm{x}_k)\|\ge 2m^2L\frac{M_2}{M_1 c}\varepsilon $$
	holds. Indeed, by Proposition \ref{prop:approxnorm2}, notice that
	we must ensure that $\|P_k^\top {\nabla}f(\bm{x}_k) \| \ge 2m^2L\frac{M_2}{M_1 c}\varepsilon + 2m^2L\varepsilon $  holds. However, by Lemma \ref{lem:aux}, this is holds with probability at least $p(\epsilon)$ if 
	$$ \dfrac{d}{3\mathcal{C}n}\|{\nabla}f(\bm{x}_k)\|  \ge  2m^2L \left(\frac{M_2}{M_1 c}+1 \right)\varepsilon ,$$
	which is exactly the condition in the statement.
\end{proof}

\textbf{Proof of Theorem \ref{nonconvex2b}}
\begin{proof}
	We have that
	\begin{equation}
		\begin{split}
			f(\bm{x}_{k+1})-f(\bm{x}_k) & {\leq} -c\alpha_k\nabla_{P_k,\varepsilon} f(\bm{x}_k)^{\top} H_k \nabla_{P_k,\varepsilon} f(\bm{x}_k)\\
			&\overset{\eqref{B_knorm}}{\leq} -c\alpha_kM_1\|\nabla_{P_k,\varepsilon} f(\bm{x}_k)\|^2.
		\end{split}
	\end{equation}
	Assume first that
	$$\|{\nabla}f(\bm{x}_k)  \| \ge  \dfrac{6\mathcal{C}nm^2L}{d} \left(\frac{M_2}{M_1 c}+1 \right)\varepsilon \quad (\ast)$$
	Assume that the event $\mathcal{A}$ holds.
	Notice that, as proved in the proof of Lemma \ref{stepsize4}, this implies that
	$$\|P_k^\top {\nabla}f(\bm{x}_k) \| \ge 2m^2L\left(\frac{M_2}{M_1 c}+1\right)\varepsilon$$
	By Proposition \ref{prop:approxnorm2}, we have that
	$$\|\nabla_{P_k,\varepsilon} f(\bm{x}_k)\| \ge \|P_k^\top \nabla f(\bm{x}_k)\| -2m^2 L \varepsilon.  $$
	Notice however that 
	$$\|P_k^\top \nabla f(\bm{x}_k)\| -2m^2 L \varepsilon \ge \frac{1}{2}\|P_k^\top \nabla f(\bm{x}_k)\|. $$
	Hence $\|\nabla_{P_k,\varepsilon} f(\bm{x}_k)\| \ge \frac{1}{2}\|P_k^\top \nabla f(\bm{x}_k)\|,$ and 
	$$f(\bm{x}_{k+1})-f(\bm{x}_k) \le -\frac{c\alpha_kM_1}{4}\|P_k^\top \nabla f(\bm{x}_k)\|^2 .$$
	By Lemma \ref{lem:aux}, and \ref{stepsize4} we deduce that, if 
	$$\|{\nabla}f(\bm{x}_k)  \| \ge  \max\left(\frac{ 4 \sqrt{d}\mathcal{C}n}{1-\epsilon}, \dfrac{6\mathcal{C}nm^2}{d} \left(\frac{M_2}{M_1 c}+1 \right) \right)L \varepsilon $$ 
	(Notice that this event $\mathcal{A}$, we have that $\|Q_k\|^2 \le \mathcal{C}n$ by Lemma \ref{rand_norm})then:
	
	\begin{equation}\label{decrease4}
		f(\bm{x}_{k+1})-f(\bm{x}_k)\leq-\frac{c\bar{\alpha} M_1}{4} \left(\dfrac{(1-\epsilon)d}{3\mathcal{C}n}\right)^2\|\nabla f(\bm{x}_k)\|^2.
	\end{equation}
	By using the same argument as before (by conditioning over $\mathcal{A}$ and using that $\mathbb{P}(\mathcal{A})\ge p(\epsilon)$), we deduce by taking the expectation that
	\begin{equation}
		\mathbb{E}[f(\bm{x}_{k+1})-f(\bm{x}_k)]\leq-\frac{c\bar{\alpha} M_1}{4} \left(\dfrac{(1-\epsilon)d}{3\mathcal{C}n}\right)^2 p(\epsilon)\mathbb{E}[\|\nabla f(\bm{x}_k)\|^2].
	\end{equation}
	That is
	$$\mathbb{E}[\|\nabla f(\bm{x}_k)\|^2] \le  \left(\frac{c\bar{\alpha} M_1}{4} \left(\dfrac{(1-\epsilon)d}{3\mathcal{C}n}\right)^2 p(\epsilon)\right)^{-1}	\mathbb{E}[f(\bm{x}_{k})-f(\bm{x}_{k+1})] . $$
	Let $K:= \max\left(\frac{ 4 \sqrt{d}\mathcal{C}n}{1-\epsilon}, \dfrac{6\mathcal{C}nm^2}{d} \left(\frac{M_2}{M_1 c}+1 \right) \right)L$
	If we do not assume that $\|{\nabla}f(\bm{x}_k)  \| \ge K \varepsilon $ holds then, we have that for any $k\ge 0$
	\begin{equation}\label{eq:ab}
		\mathbb{E}[\|\nabla f(\bm{x}_k)\|^2] \le  \left(\frac{c\bar{\alpha} M_1 p(\epsilon)}{4} \left(\dfrac{(1-\epsilon)d}{3\mathcal{C}n}\right)^2 \right)^{-1}	\mathbb{E}[f(\bm{x}_{k})-f(\bm{x}_{k+1})] + K^2\varepsilon^2. 
	\end{equation}
	Adding this equation to $k=0,1,\dots$, we obtain 
	
	\begin{equation}
		\sum_{j=0}^{k}\mathbb{E}[\|\nabla f(\bm{x}_j)\|^2]\leq  \left(\frac{c\bar{\alpha} M_1 p(\epsilon)}{4} \left(\dfrac{(1-\epsilon)d}{3\mathcal{C}n}\right)^2 \right)^{-1}(f(\bm{x}_0)-f^*) + (k+1)K^2\varepsilon^2.
		\notag
	\end{equation}
	From
	\begin{equation}
		\begin{split}
			\mathbb{E}\left[\min_{0\leq j\leq k}\|\nabla f(\bm{x}_k)\|^2\right]&\leq\dfrac{1}{k+1}\sum_{j=0}^{k}\mathbb{E}[\|\nabla f(\bm{x}_k)\|^2], \\
		\end{split}
	\end{equation}
	we deduce that
	$$\mathbb{E}\left[\min_{0\leq j\leq k}\|\nabla f(\bm{x}_k)\|^2\right] \le\dfrac{1}{k+1} \left(\frac{c\bar{\alpha} M_1 p(\epsilon)}{4} \left(\dfrac{(1-\epsilon)d}{3\mathcal{C}n}\right)^2 \right)^{-1}(f(\bm{x}_0)-f^*) + K^2\varepsilon^2.$$
\end{proof}

\textbf{Proof of Theorem \ref{convex3}}
\begin{proof}
	From the convexity of $f$ and the definition of $R_0$, for any $\bm{x}\in\mathbb{R}^n$ and the optimal solution $\bm{x}^*$
	of \eqref{R_0}, 
	\begin{equation}
		f(\bm{x})-f^*\leq\nabla f(\bm{x})^\top(\bm{x}-\bm{x}^*)\leq\|\nabla f(\bm{x})\| \|\bm{x}-\bm{x}^*\|\leq\|\nabla f(\bm{x})\|R_0,
		\notag
	\end{equation}
	that is, 
	\begin{equation}
		\|\nabla f(\bm{x})\|^2\geq\dfrac{(f(\bm{x})-f^*)^2}{R_0^2}.
	\end{equation}
	Now we define
	\begin{equation}
		\Delta_k:=\mathbb{E}[f(\bm{x}_k)-f^*].
		\notag
	\end{equation}
	From Jensen's inequality, it follows that
	\begin{equation}
		\mathbb{E}[(f(\bm{x}_k)-f^*)^2]\geq\Delta_k^2.
		\notag
	\end{equation}
	Setting $\bm{x}=\bm{x}_k$ and taking the expectation of both sides of \eqref{f(x)}, we obtain
	\begin{equation}
		\mathbb{E}[\|\nabla f(\bm{x}_k)\|^2]\geq\dfrac{\mathbb{E}[(f(\bm{x}_k)-f^*)^2]}{R_0^2}\geq\dfrac{\Delta_k^2}{R_0^2}
	\end{equation}
	Assume first that 
	\begin{equation}
		f(x_k)-f^*\ge 2R_0K\varepsilon.
	\end{equation}
	Then we also have that
	\begin{equation}
		\mathbb{E}[\|\nabla f(\bm{x}_k)\|^2]\ge 4K^2\varepsilon^2,
	\end{equation}
	the above implies that
	\begin{equation}\label{eq:abc}
		\|\nabla f(\bm{x}_k )\|^2-K^2\varepsilon^2 \ge \frac{3}{4}\|\nabla f(\bm{x}_k )\|^2
	\end{equation}
	Rewriting \eqref{eq:ab} using $\Delta_k$, 
	we obtain
	
	\begin{eqnarray*}
		\Delta_{k+1} & \leq & \Delta_k-\left(\frac{c\bar{\alpha} M_1 p(\epsilon)}{4} \left(\dfrac{(1-\epsilon)d}{3\mathcal{C}n}\right)^2 \right)(\mathbb{E}[\|\nabla f(\bm{x}_k )\|^2]-K^2\varepsilon^2) \\
		& \overset{\eqref{eq:abc}}{\leq} & \Delta_k-\frac{3}{4}\left(\frac{c\bar{\alpha} M_1 p(\epsilon)}{4} \left(\dfrac{(1-\epsilon)d}{3\mathcal{C}n}\right)^2 \right)\mathbb{E}[\|\nabla f(\bm{x}_k )\|^2]\\
		& \overset{\eqref{minusnabla}}{\leq} & \Delta_k-\frac{3}{4}\left(\frac{c\bar{\alpha} M_1 p(\epsilon)}{4} \left(\dfrac{(1-\epsilon)d}{3\mathcal{C}R_0n}\right)^2 \right)\Delta_k^2
		\notag
	\end{eqnarray*}
	Dividing this equation by both sides $\Delta_k\Delta_{k+1}(>0)$ and using the fact that $\Delta_k$ is monotonically decreasing with respect to $k$ and therefore $\dfrac{\Delta_k}{\Delta_{k+1}}\geq 1$, 
	\begin{equation}
		\dfrac{1}{\Delta_k}\leq\dfrac{1}{\Delta_{k+1}}-\frac{3}{4}\left(\frac{c\bar{\alpha} M_1 p(\epsilon)}{4} \left(\dfrac{(1-\epsilon)d}{3\mathcal{C}R_0n}\right)^2 \right)\dfrac{\Delta_k }{\Delta_{k+1}}\leq\dfrac{1}{\Delta_{k+1}}-\frac{3}{4}\left(\frac{c\bar{\alpha} M_1 p(\epsilon)}{4} \left(\dfrac{(1-\epsilon)d}{3\mathcal{C}R_0n}\right)^2 \right).
		\notag
	\end{equation}
	By using this with $k=0,1,\dots$, we get
	\begin{equation}
		\dfrac{1}{\Delta_0}\leq\dfrac{1}{\Delta_{k+1}}-\frac{3}{4}\left(\frac{c\bar{\alpha} M_1 p(\epsilon)}{4} \left(\dfrac{(1-\epsilon)d}{3\mathcal{C}R_0n}\right)^2 \right) (k+1)\bar{\alpha} 
		\notag
	\end{equation}
	By writing this, 
	\begin{equation}\label{Deltaa}
		\begin{split}
			\Delta_{k+1}\leq 
			\dfrac{146\Delta_0\mathcal{C}^2n^2R_0^2}{146\mathcal{C}^2n^2R_0^2+3\Delta_0cM_1p(\epsilon)(1-\epsilon)^2d^2(k+1)\bar{\alpha }}.
		\end{split}
	\end{equation}
	Now, from the definition of $\Delta_k$, \eqref{Lipdif} and \eqref{R_0} together with the convexity of $f$ and
	$\nabla f(\bm{x}^*)=\bm{0}$, we have
	\begin{equation}\label{Delta_1}
		\Delta_0=f(\bm{x}_0)-f^*\leq L \|\bm{x}_0-\tilde{\bm{x}}\|^2 \leq L R_0^2,
	\end{equation}
	where
	\begin{equation}
		\tilde{\bm{x}}:=\argmin_{\bm{x}^* \in \Omega}\|\bm{x}_0-\bm{x}^*\|.
		\notag
	\end{equation}
	Since the term in the right of inequalities \eqref{Deltaa} is monotonically increasing with respect to $\Delta_0$,
	\eqref{Deltaa} together with  \eqref{Delta_1} leads to
	\begin{equation}
		\Delta_{k}\leq 
		\dfrac{L 146\mathcal{C}^2n^2R_0^2}{146\mathcal{C}^2n^2+L cM_1p(\epsilon)(1-\epsilon)^2d^2k\bar{\alpha}},
		\notag
	\end{equation}
	We complete the proof, by adding the term $2R_0K\varepsilon$ to the bound in the case where
	\begin{equation}
		f(x_k)-f^*\ge 2R_0K\varepsilon.
	\end{equation}
	is not satisfied 
\end{proof}

\section{Additional Numerical Experiments}\label{sec:add_numerics}

\begin{figure*}[!htbp]
	\centering
	\subfloat[Objective function values for 10 runs with $d=10$ and $m=10$ (left) or $m=100$ (right) \label{fig:function-values-benchmark}]{
		\includegraphics[width=0.50\linewidth]{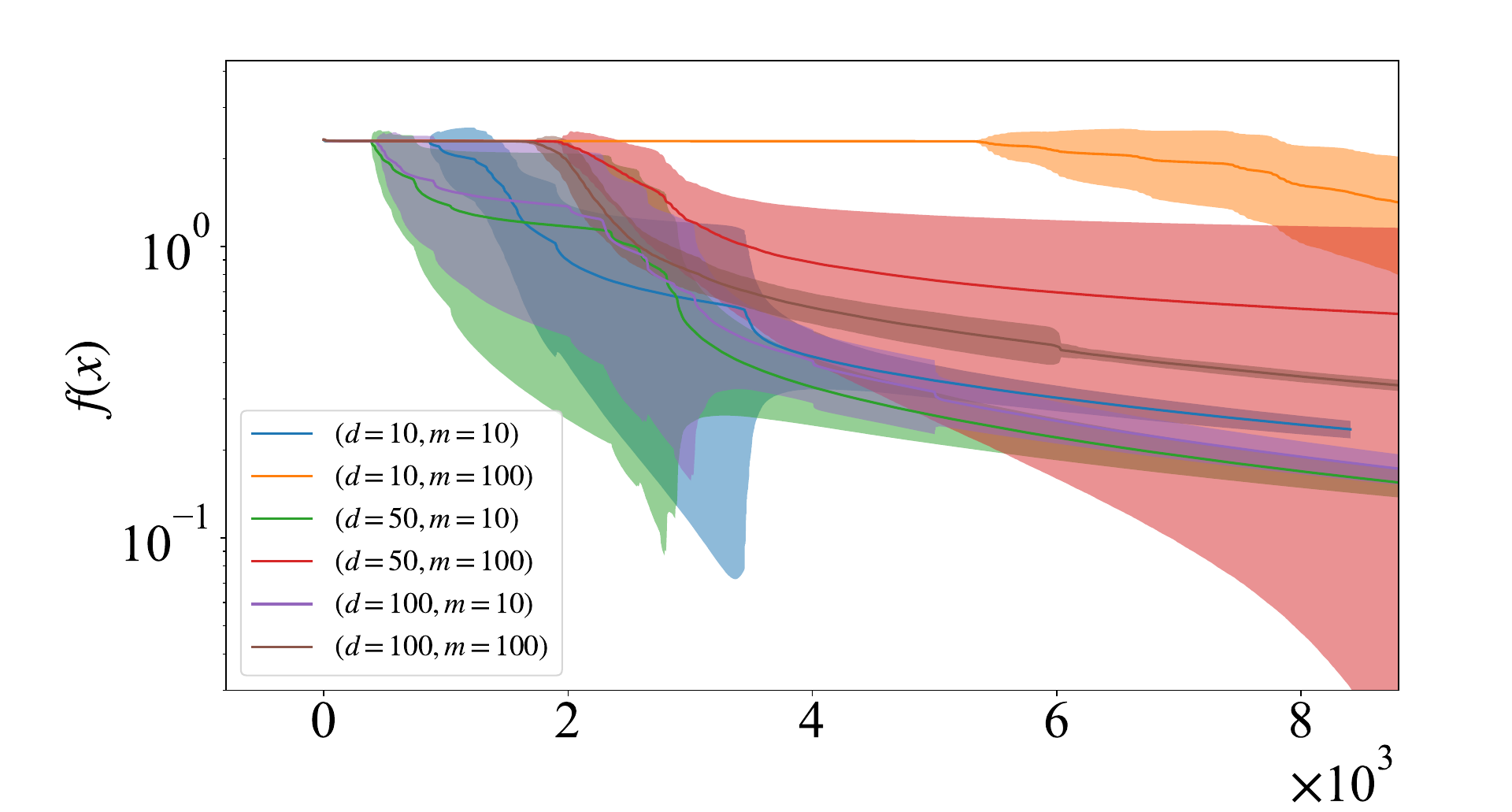}
		\includegraphics[width=0.50\linewidth]{figures/new_experiment/compare_proposed_parameter_mlpnet_average.pdf}
	}\par\medskip
	\subfloat[Objective function values for 10 runs with $d=50$ and $m=10$ (left) or $m=100$ (right) \label{fig:gradnorm-benchmark}]{
		\includegraphics[width=0.50\linewidth]{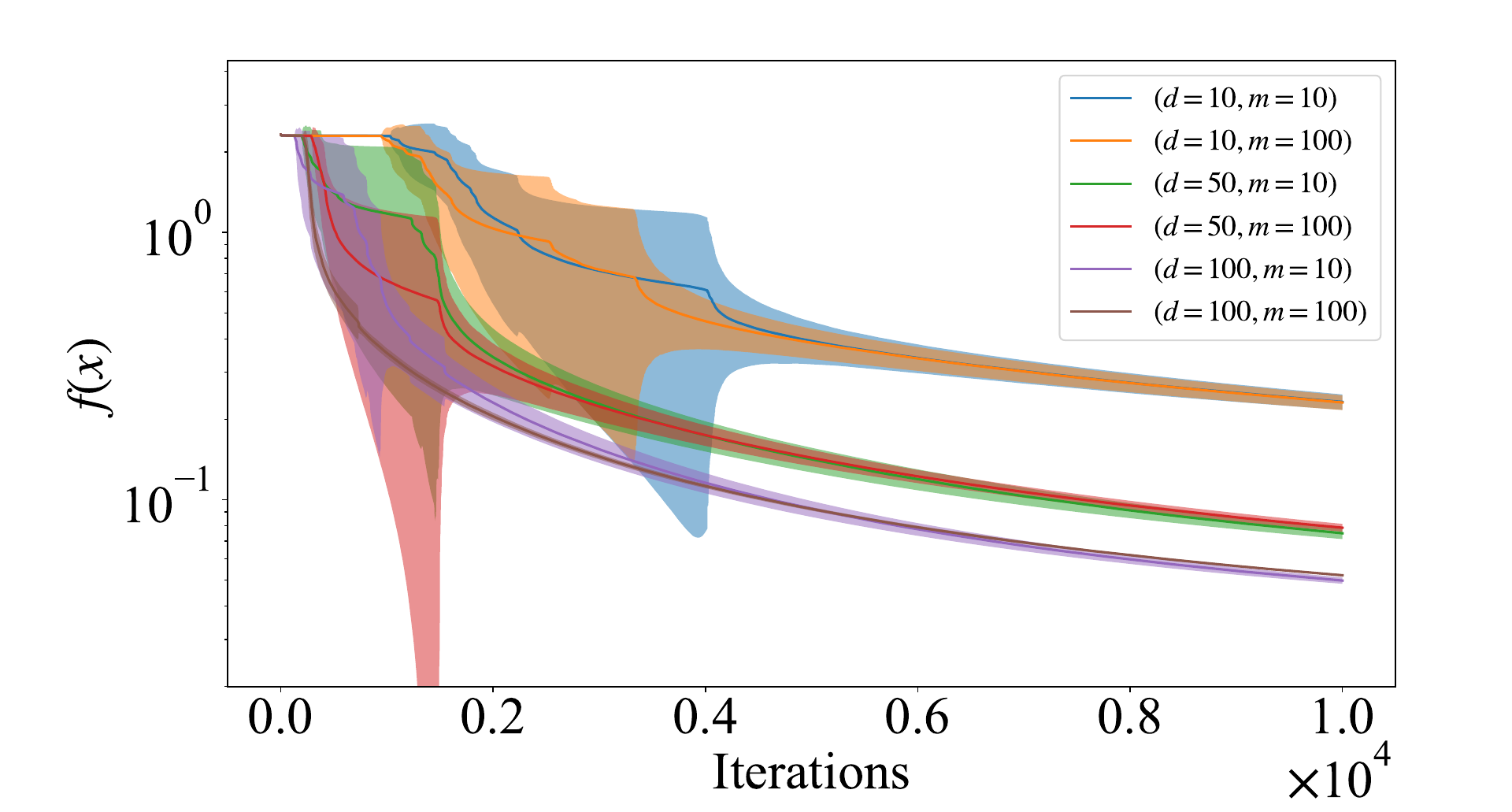}
		\includegraphics[width=0.50\linewidth]{figures/new_experiment/compare_proposed_parameter_mlpnet_iteration_average.pdf}
	}\par\medskip
	\caption{The change of performance depending on $m$ and $d$ for the LNN}\label{fig:parameter_m_d}
\end{figure*}

\begin{figure*}[!htbp]
	\centering
	\subfloat[Objective function values for 10 runs with $d=10$ and $m=10$ (left) or $m=100$ (right) \label{fig:parameter_d10}]{
		\includegraphics[width=0.50\linewidth]{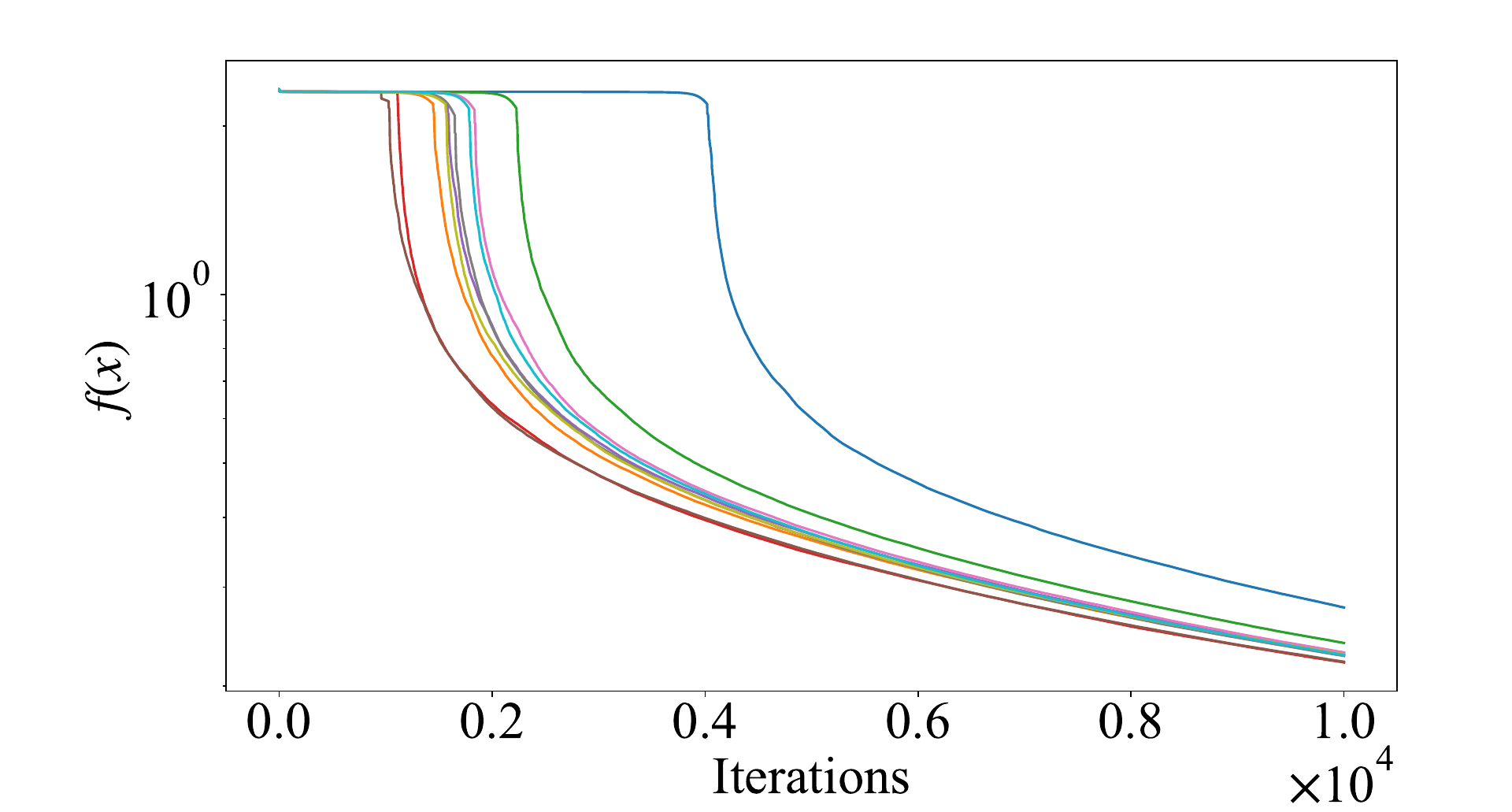}
		\includegraphics[width=0.50\linewidth]{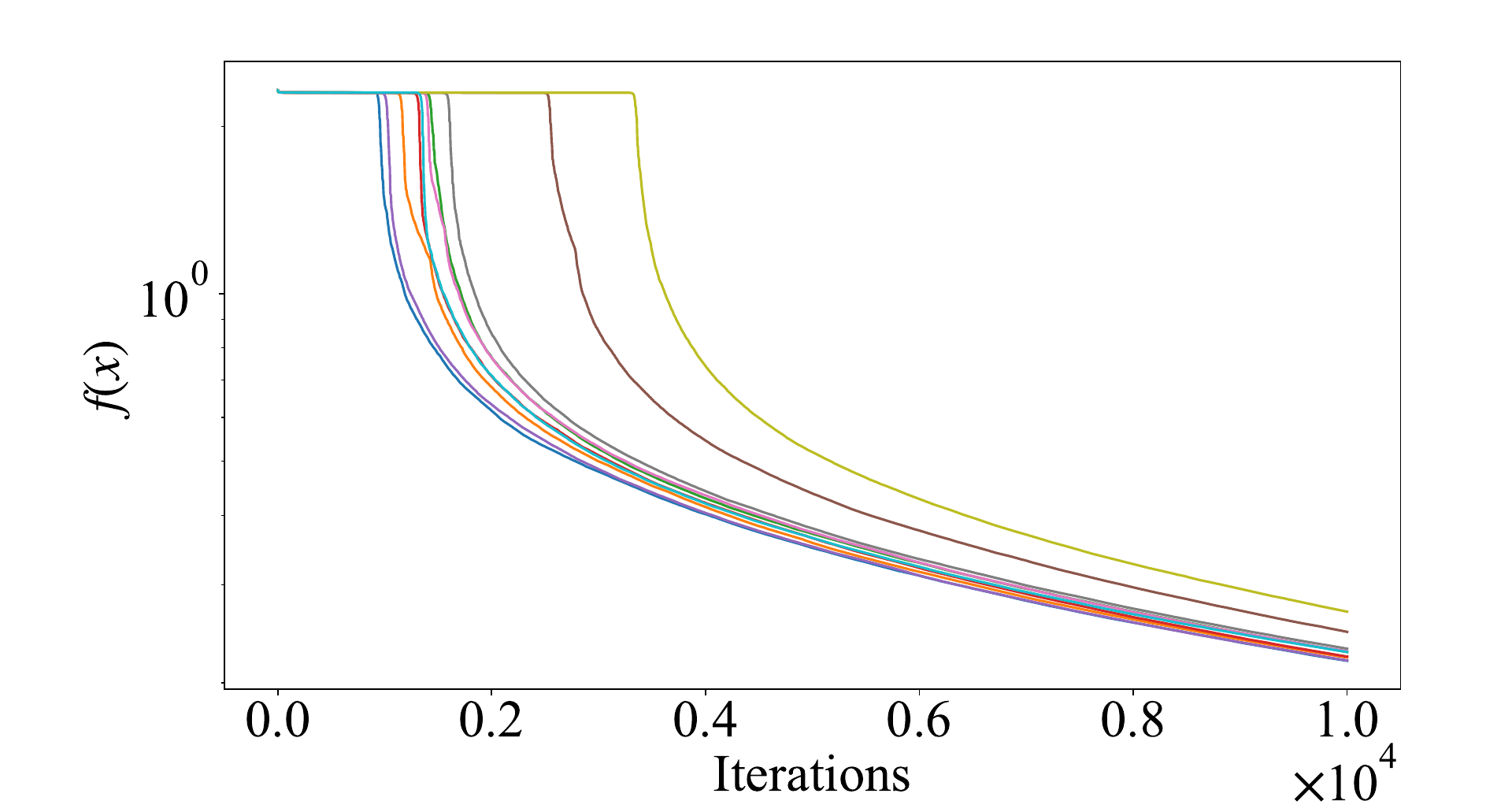}
	}\par\medskip
	\subfloat[Objective function values for 10 runs with $d=50$ and $m=10$ (left) or $m=100$ (right) \label{fig:parameter_d50}]{
		\includegraphics[width=0.50\linewidth]{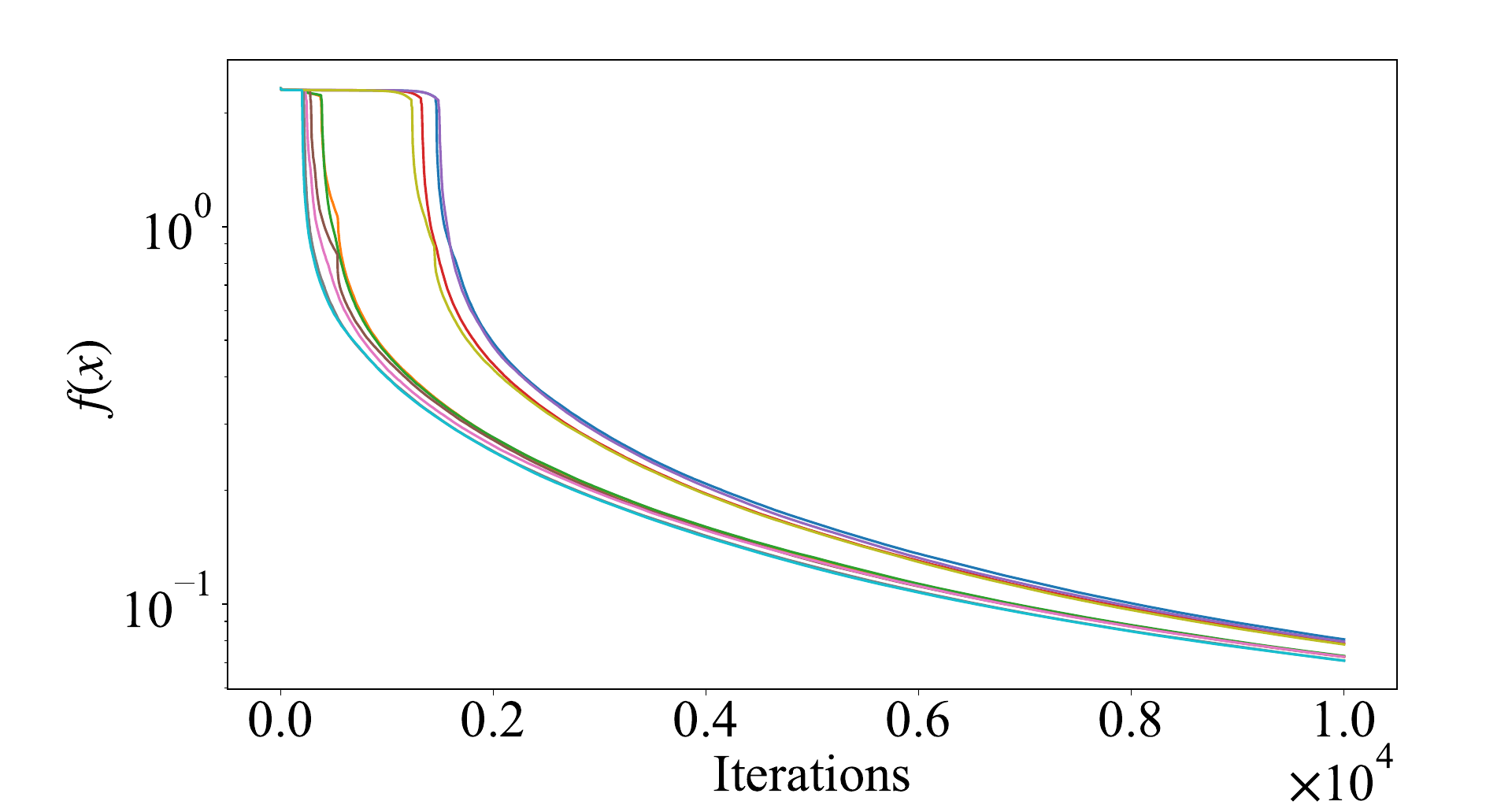}
		\includegraphics[width=0.50\linewidth]{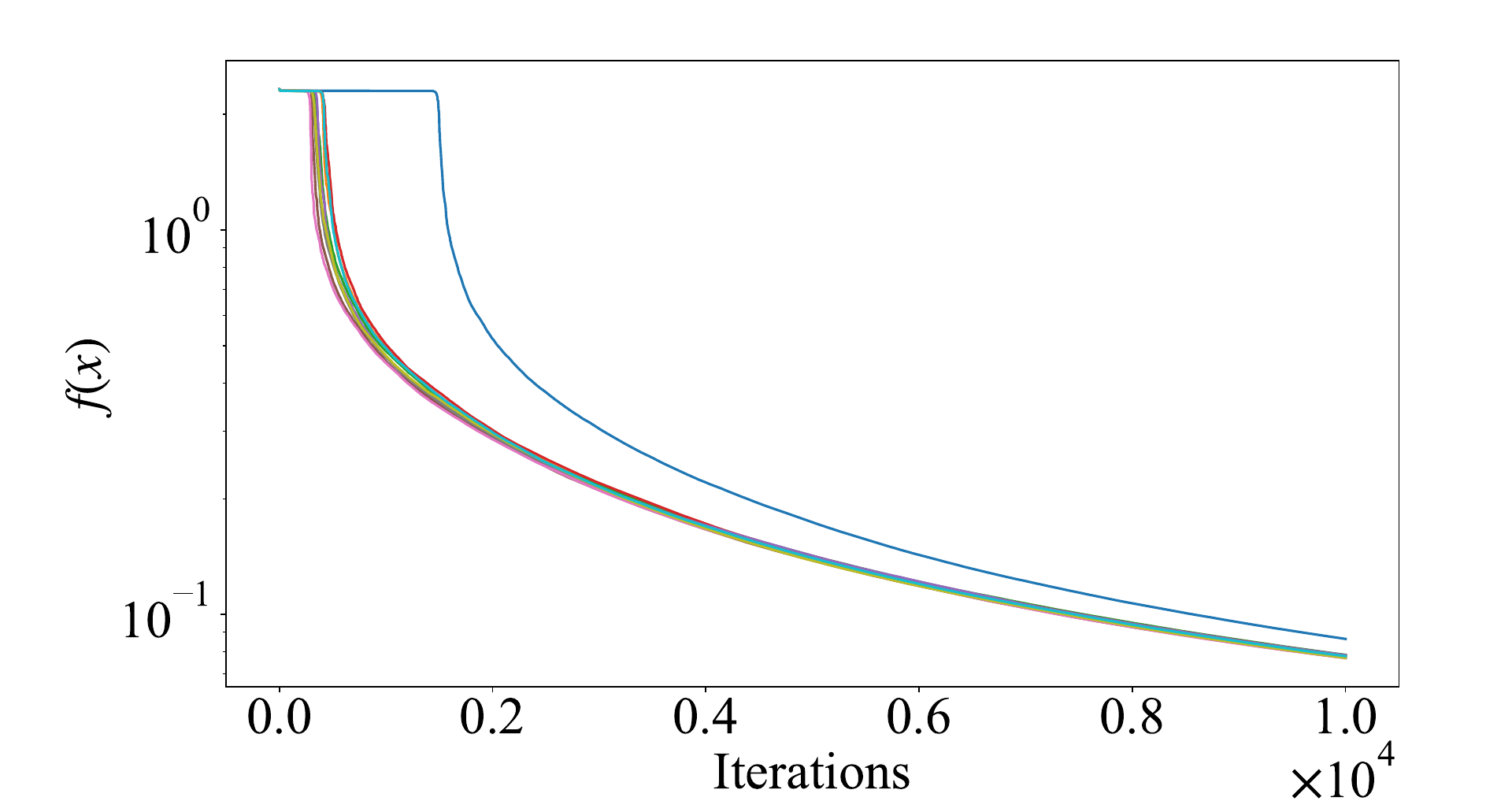}
	}\par\medskip
	\subfloat[Objective function values for 10 runs with $d=100$ and $m=10$ (left) or $m=100$ (right) \label{fig:parameter_d100}]{
		\includegraphics[width=0.50\linewidth]{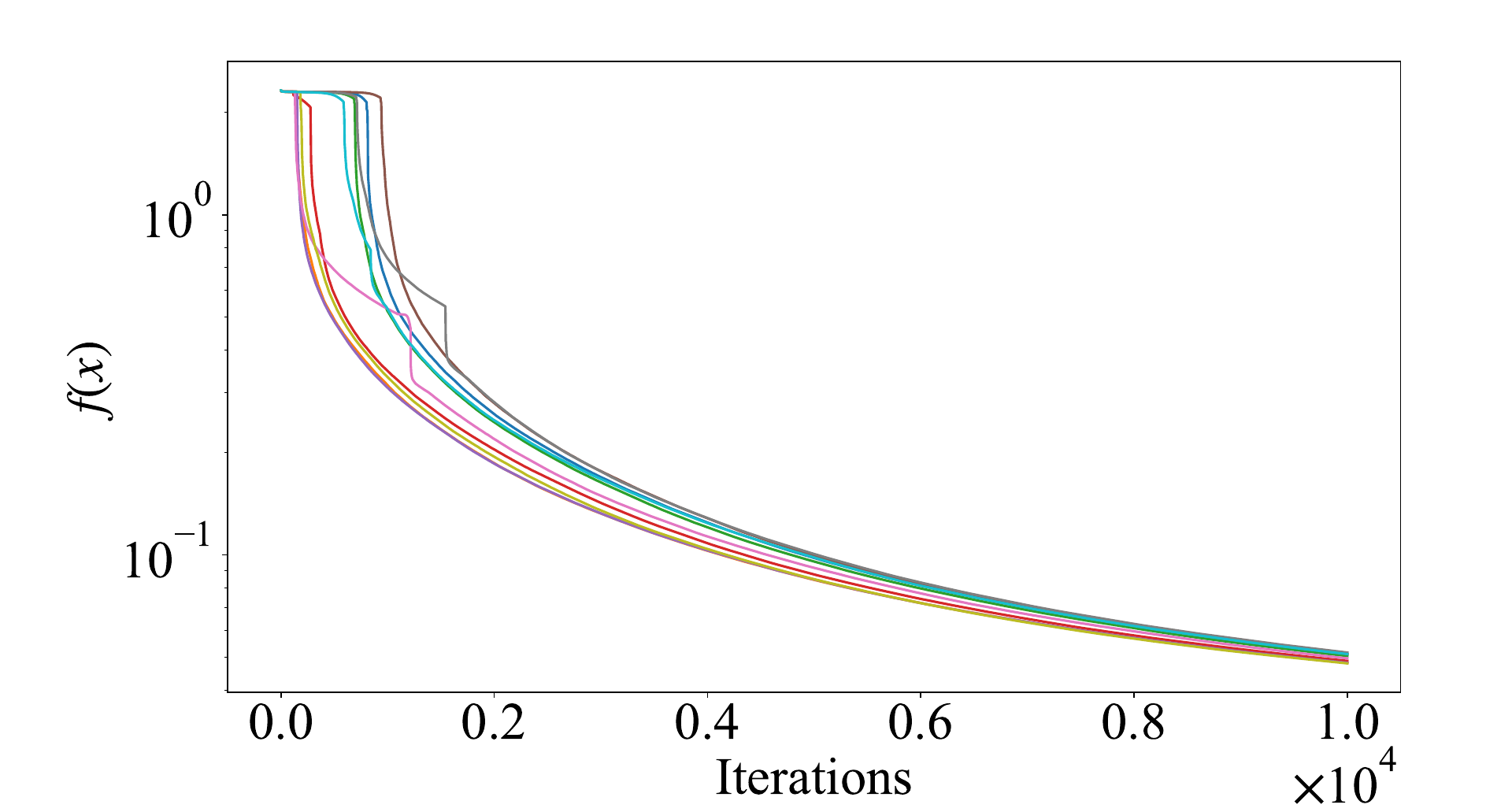}
		\includegraphics[width=0.50\linewidth]{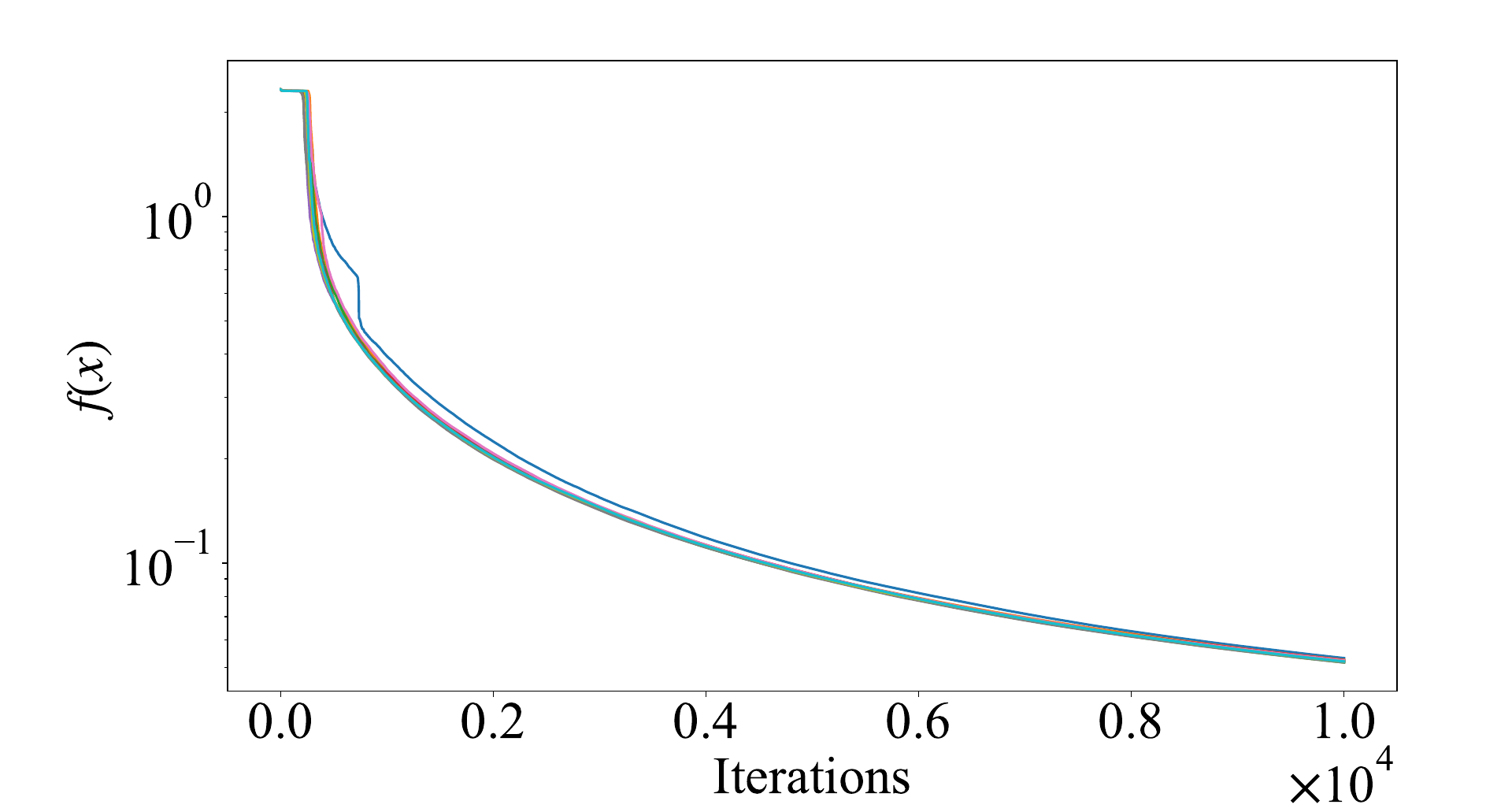}
	}\par\medskip
	\caption{The performance of 10 runs for each $m$ and $d$ for the LNN}\label{fig:parameter_m_d_allrun}
\end{figure*}

\begin{figure*}[!htbp]
	\centering
	\subfloat[The average  $\pm$ the standard deviation (left) and the average (right) of objective function values for 10 runs  \label{fig:function-values-errorbar_mlpnet}]{
		\includegraphics[width=0.50\linewidth]{figures/new_experiment/compare_existing_small_mlpnet.pdf}
		\includegraphics[width=0.50\linewidth]{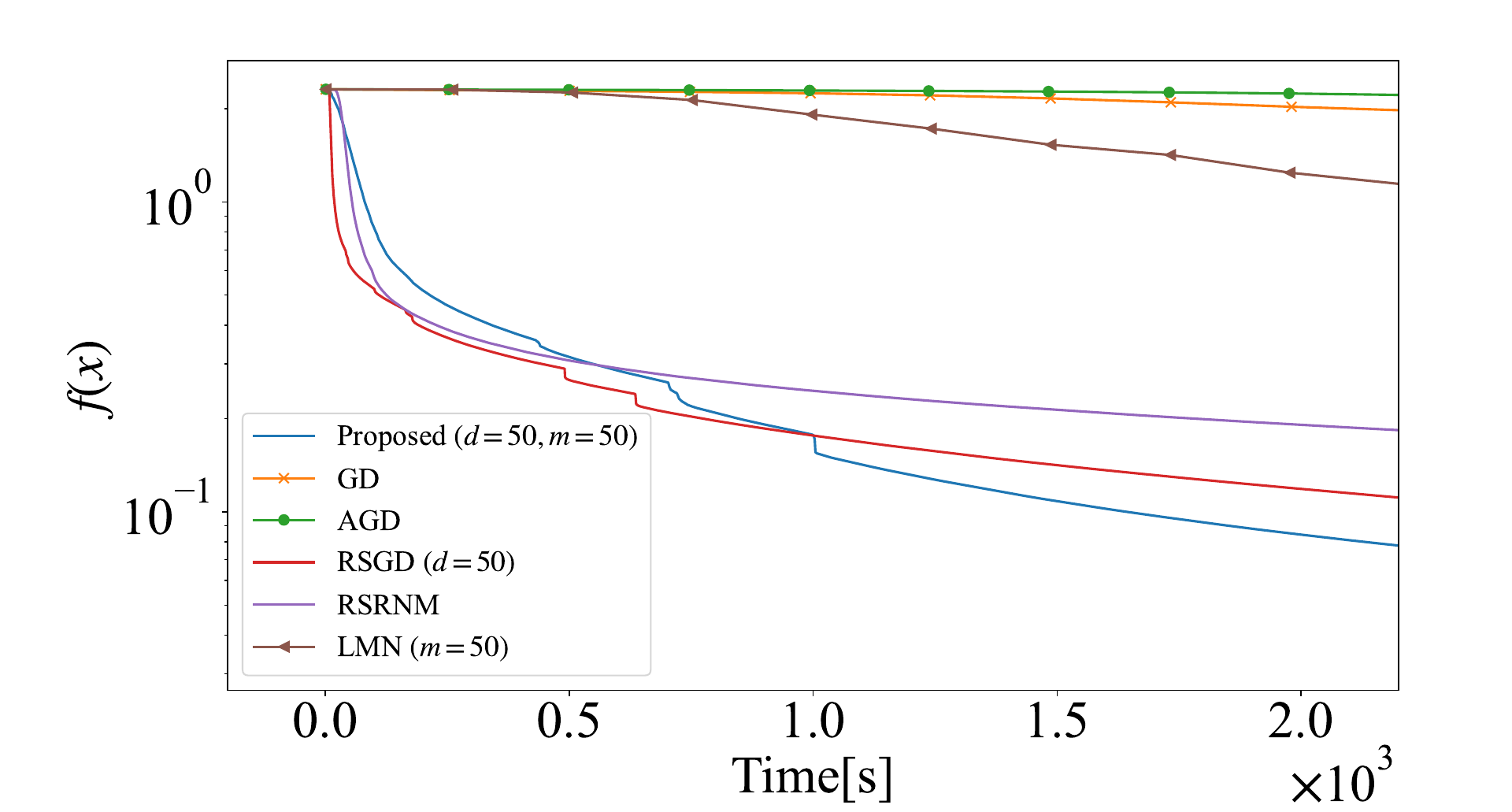}
	}\par\medskip
	\subfloat[The average  $\pm$ the standard deviation (left) and the average (right) of gradient norms for 10 runs       \label{fig:gradnorm-errorbar_mlpnet}]{
		\includegraphics[width=0.50\linewidth]{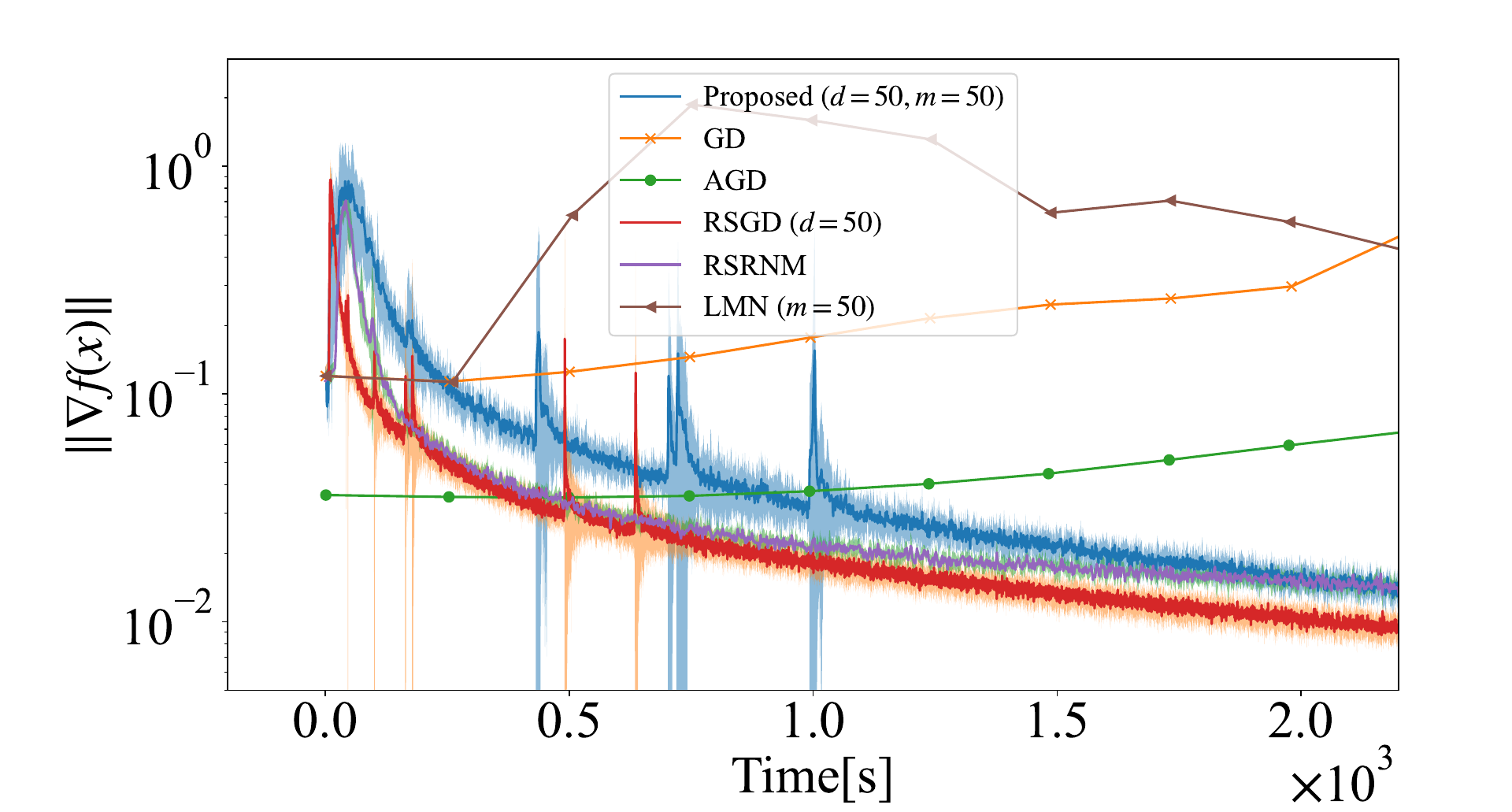}
		\includegraphics[width=0.50\linewidth]{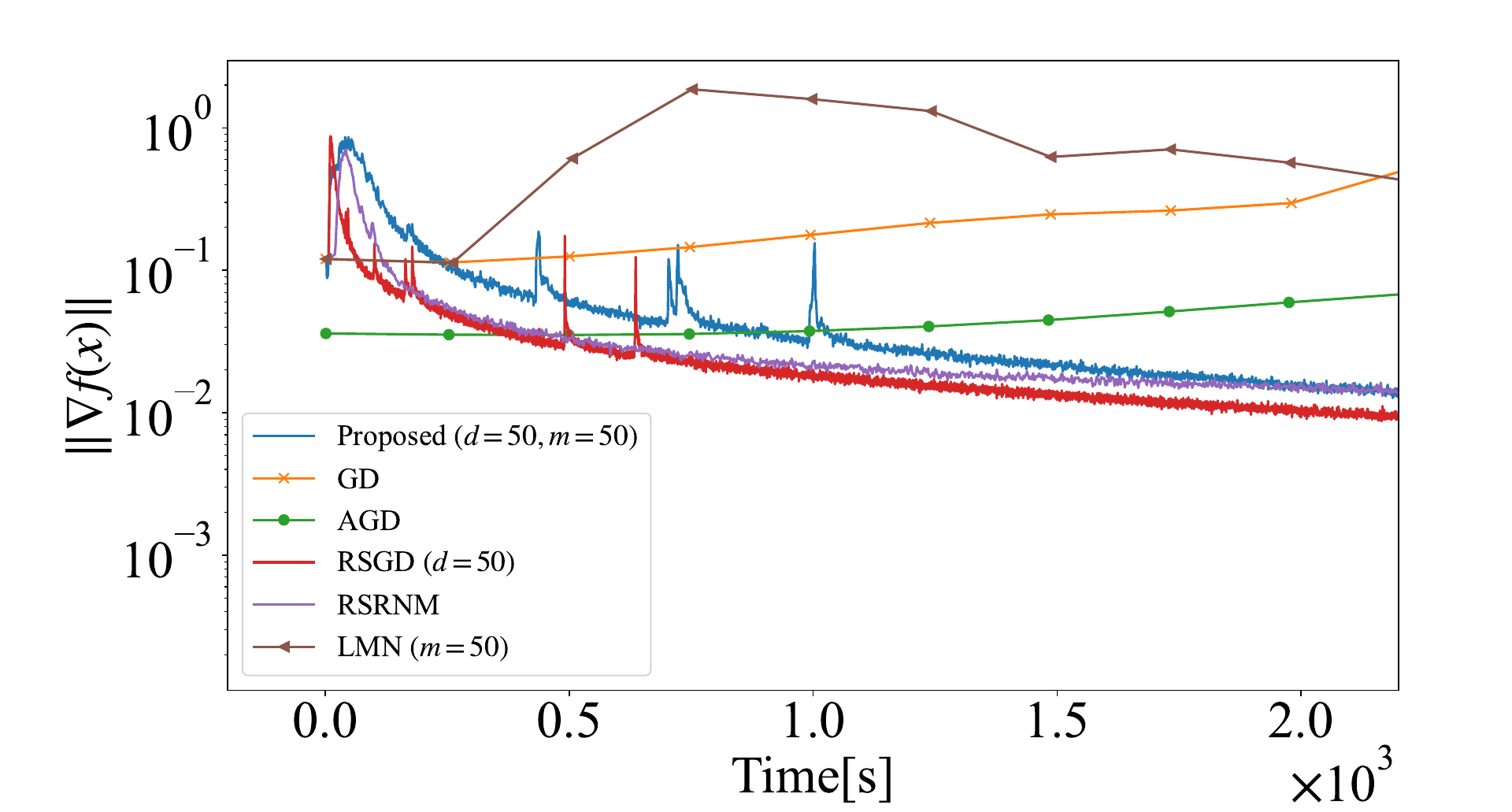}
	}\par\medskip
	\caption{Comparison of our method to existing methods for the LNN. The average $\pm$ the standard deviation for 10 runs (left) and only the average (right) are shown for the objective function value (top) and the gradient norm (bottom).}\label{fig:errorbar_mlpnet}
\end{figure*}

\begin{figure*}[!htbp]
  \centering
	\subfloat[The average  $\pm$ the standard deviation (left) and the average (right) of objective function values for 5 runs \label{fig:function-values-benchmark1}]{
		\includegraphics[width=0.50\linewidth]{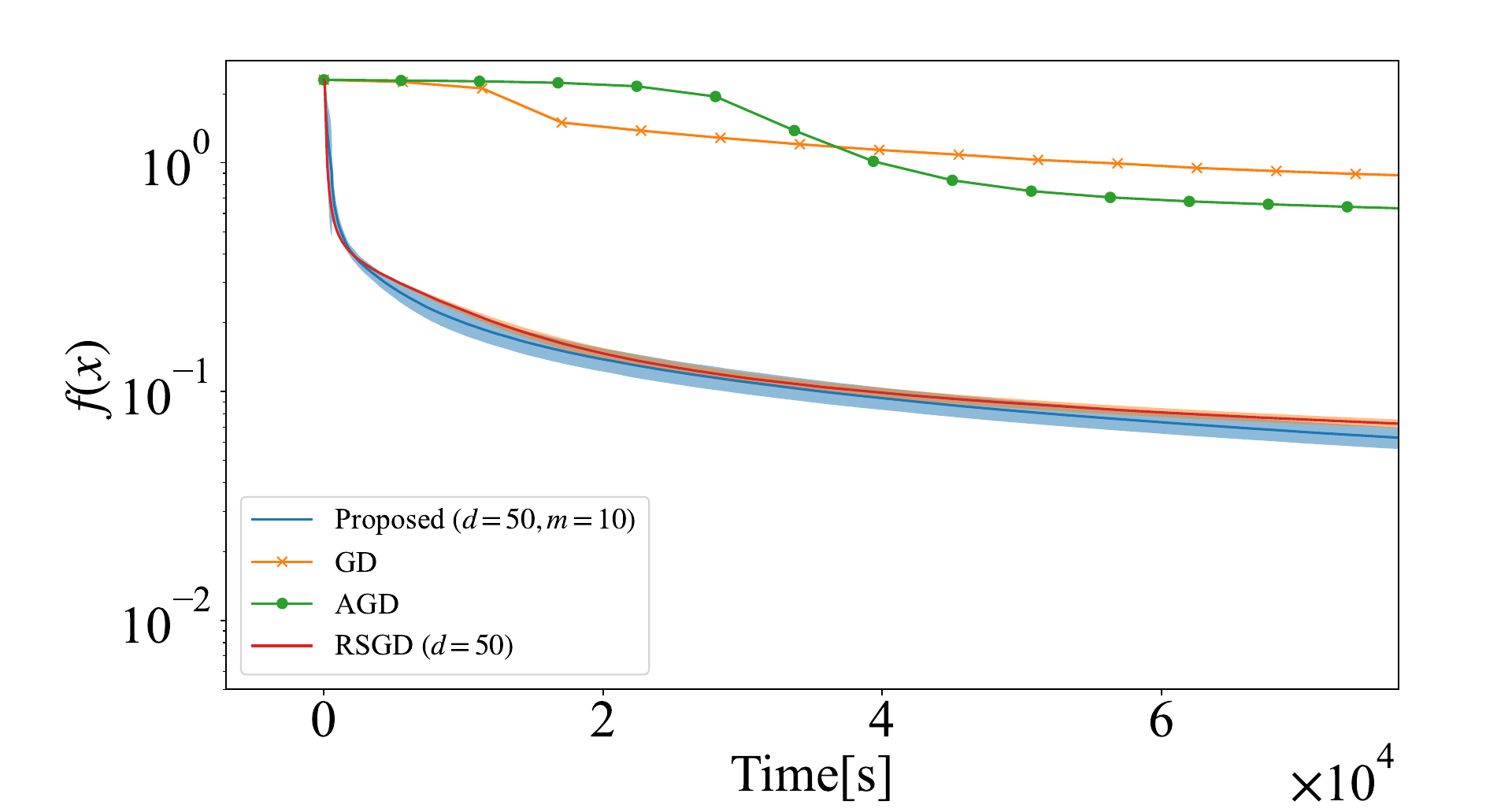}
		\includegraphics[width=0.50\linewidth]{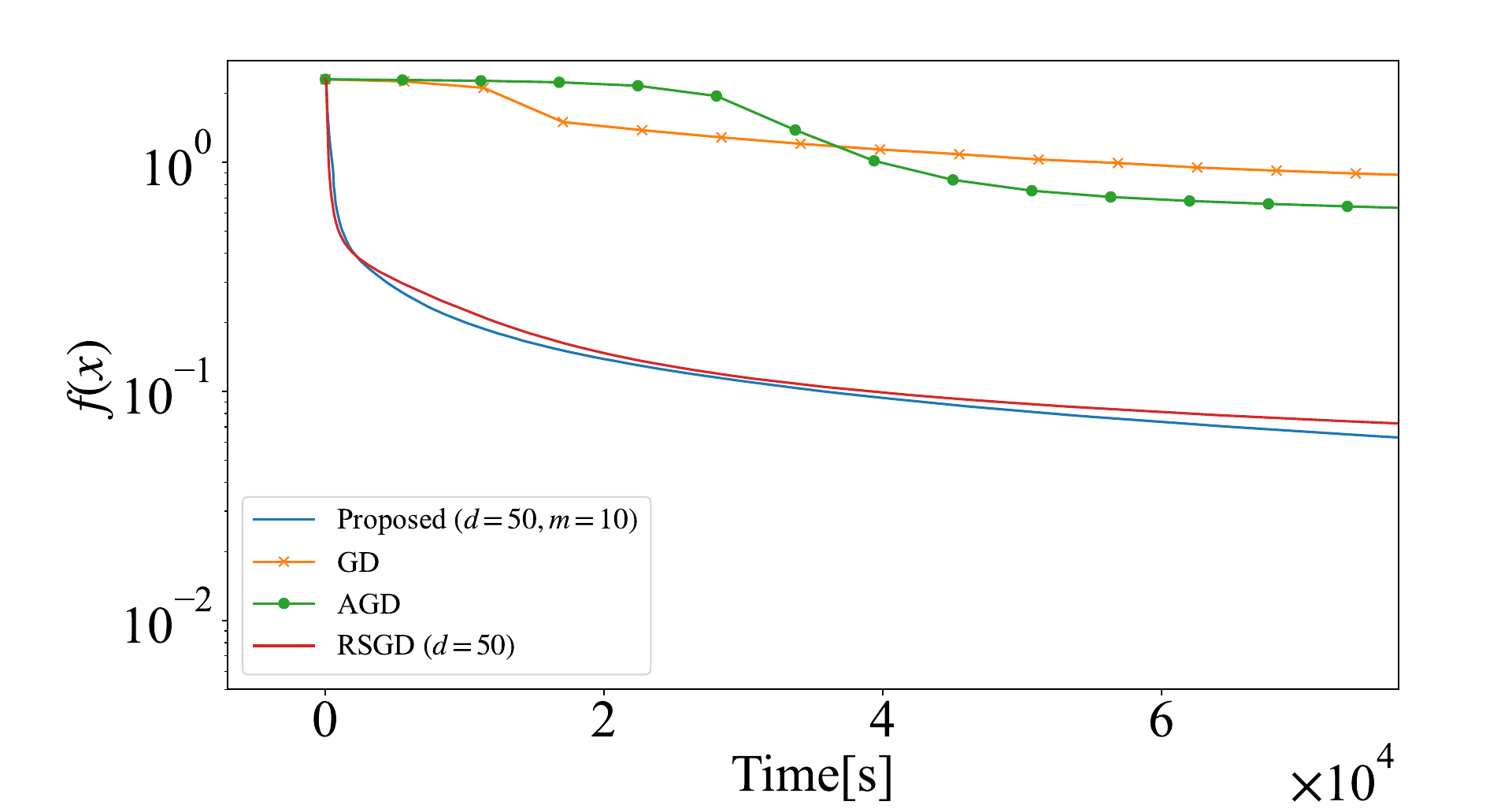}
	}\par\medskip
	\subfloat[The average  $\pm$ the standard deviation (left) and the average (right) of gradient norms for 5 runs       \label{fig:gradnorm-benchmark1}]{
		\includegraphics[width=0.50\linewidth]{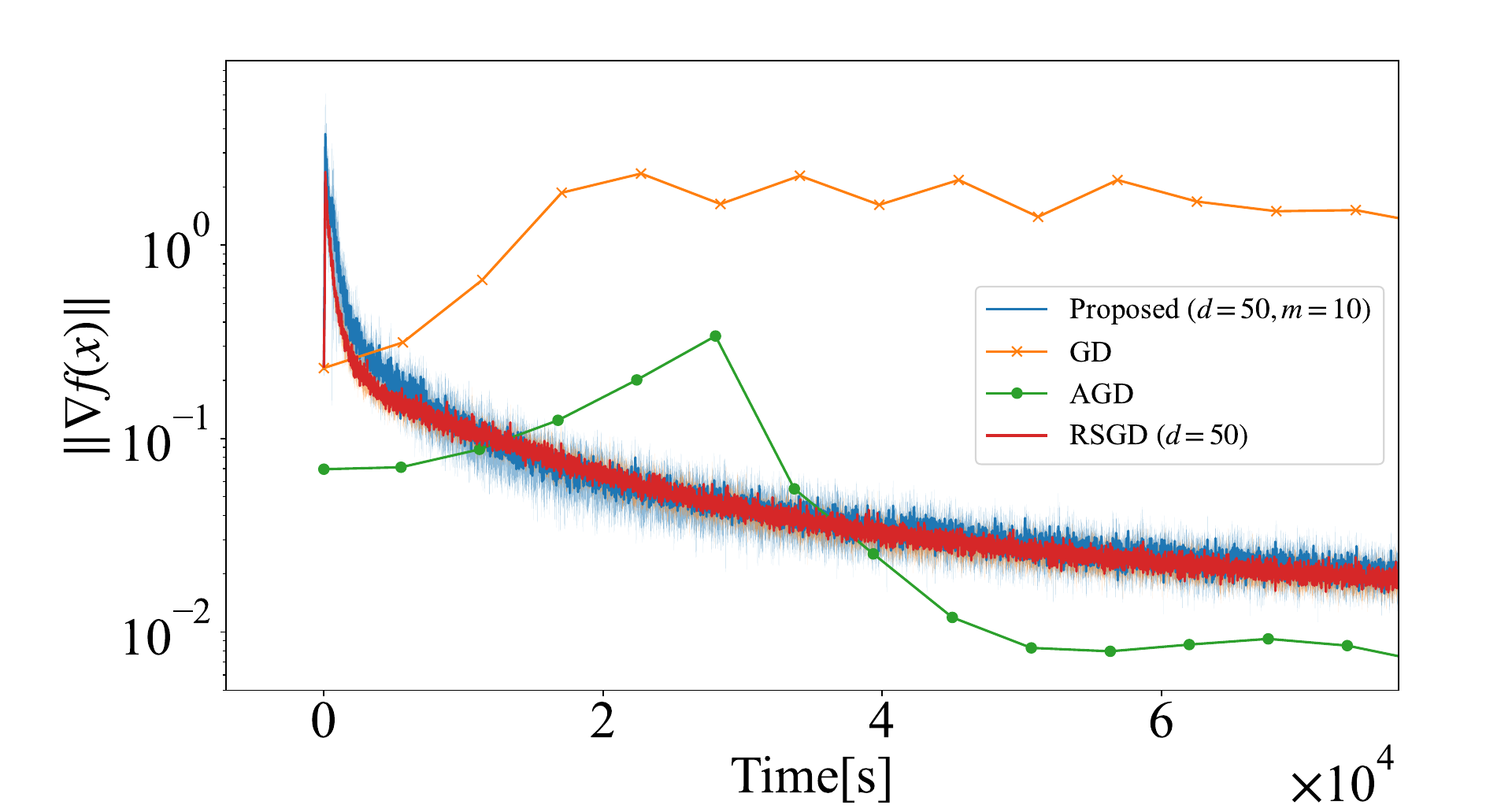}
		\includegraphics[width=0.50\linewidth]{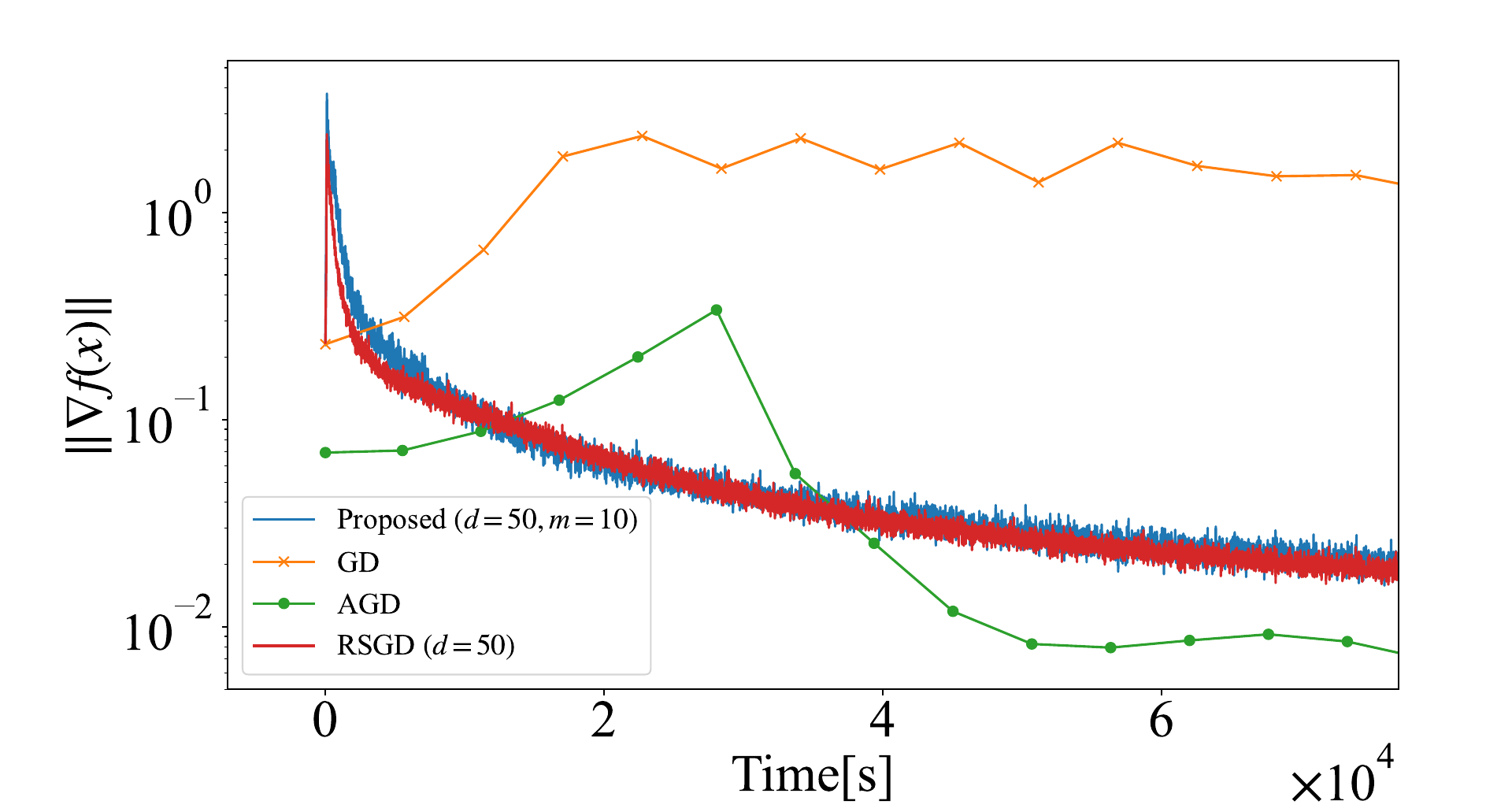}
	}\par\medskip
	\caption{Comparison of our method to existing methods for the CNN. The average $\pm$ the standard deviation for 5 runs (left) and only the average (right) are shown for the objective function value (top) and the gradient norm (bottom).}\label{fig:errorbar_cnn}
\end{figure*}

\begin{figure}[tb]
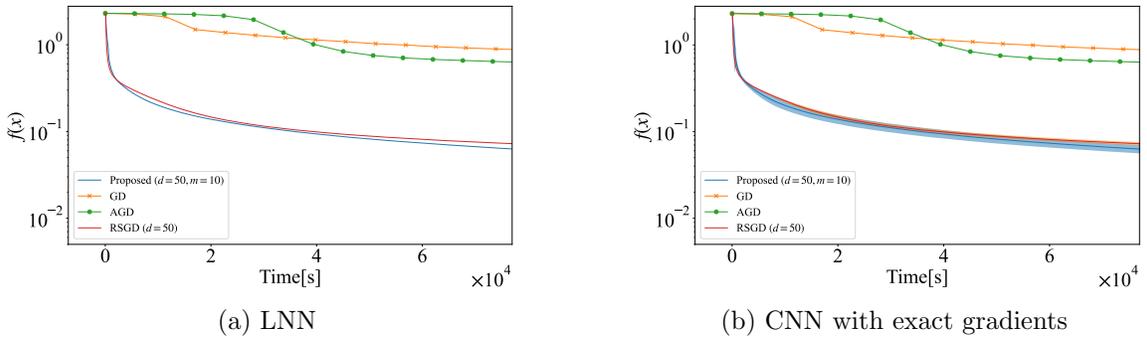

	\begin{minipage}[t]{0.5\hsize}
		\centering
		\includegraphics[width=3.0 in]{figures/new_experiment/compare_cnn_existing_average.pdf}
		\subcaption{LNN}
		\label{fig:result_cnn_ad_average}
	\end{minipage} 
	\begin{minipage}[t]{0.5\hsize}
		\centering
		\includegraphics[width=3.0 in]{figures/new_experiment/compare_cnn_existing.pdf}
		\subcaption{CNN with exact gradients}
		\label{fig:result_cnn_ad}
	\end{minipage} 
	\caption{The average $\pm$ the standard deviation for CNN (5 runs) with exact gradients.}
	\label{fig:realworld_data_ad}
\end{figure}


\clearpage

\end{document}